\documentclass{amsart}

\usepackage{amscd}
\usepackage[all]{xy}
\usepackage{slashed}
\usepackage{todonotes}

\usepackage{comment}

\usepackage{hyperref}

\newtheorem{thm}{Theorem}[section]

\newtheorem{question}{Question}

\theoremstyle{remark}
\newtheorem{rem}{Remark}[section]

\newcommand{\Z}{\mathbb{Z}}

\newcommand{\R}{\mathbb{R}}
\newcommand{\C}{\mathbb{C}}

\newcommand{\CP}{\mathbb{CP}}

\newcommand{\spc}{\mathrm{spin}^c}

\newcommand{\Gr}{\mathrm{Gr}}

\newcommand{\fraks}{\mathfrak{s}}
\newcommand{\frakt}{\mathfrak{t}}

\newcommand{\PSC}{\mathrm{PSC}}

\newcommand{\Pin}{\mathrm{Pin}}
\newcommand{\Diff}{\mathrm{Diff}}
\newcommand{\Homeo}{\mathrm{Homeo}}
\newcommand{\Symp}{\mathrm{Symp}}
\newcommand{\Aut}{\mathrm{Aut}}

\newcommand{\Th}{\mathrm{Th}}

\newcommand{\inc}{\hookrightarrow}

\newcommand{\Ker}{\mathop{\mathrm{Ker}}\nolimits}
\newcommand{\Coker}{\mathop{\mathrm{Coker}}\nolimits}

\newcommand{\rank}{\mathop{\mathrm{rank}}\nolimits}

\newcommand{\id}{\mathrm{id}}
\newcommand{\ind}{\mathop{\mathrm{ind}}\nolimits}
\providecommand\sslash{\mathbin{/\mkern-5.5mu/}}

\numberwithin{equation}{section}

\begin{document}

\title{AMS Journal Sample}

\author{Hokuto Konno}
\curraddr{Graduate School of Mathematical Sciences, the University of Tokyo, 3-8-1 Komaba, Meguro, Tokyo 153-8914, Japan}
\email{konno@ms.u-tokyo.ac.jp}
\thanks{The author is grateful to Mikio Furuta, Nobuhiro Nakamura, and Masaki Taniguchi, as well as the referees, for reading the draft and offering valuable comments. The author was supported in part by JSPS KAKENHI Grant Numbers 25K00908.}

\dedicatory{Dedicated to Mikio Furuta}


\title{Gauge theory for families}

\maketitle

\begin{abstract}
This article surveys gauge theory for families and its applications to the comparison between the diffeomorphism group and the homeomorphism group of $4$-manifolds, up to 2021.
\end{abstract}

\section{Introduction: $\Diff$ vs. $\Homeo$ in Dimension Four}
\label{sec1}

This article surveys {\it gauge theory for families} up to 2021\footnote{This is a translation of a Japanese article originally written in 2021. In what follows, for papers published after 2021, we mention only those references that are directly related to the topics discussed in the original 2021 version of this survey; post-2021 developments are treated in \cite{KonnoRonsetsu2}.}, which develops gauge theory for smooth families of 4-manifolds, and its applications to comparing the diffeomorphism groups and the homeomorphism groups of 4-manifolds.
In this section, we first formulate the concrete problems we wish to study and place them in the context of phenomena in other dimensions.

\subsection{Main question}
Many of the results presented in this article can be stated as answers to the following question~\ref{main question}:

\begin{question}
\label{main question}
Let $X$ be a smooth $4$-manifold.
Consider the natural inclusion map from the diffeomorphism group to the homeomorphism group
\begin{align*}
\Diff(X) \inc \Homeo(X)
\end{align*}
Is this a weak homotopy equivalence?
If not, for which values of $n$ is the induced map on homotopy groups
\begin{align}
\label{eq: main comparison map on htpy grp}
\pi_{n}(\Diff(X)) \to \pi_{n}(\Homeo(X))
\end{align}
not an isomorphism?
And when it fails to be an isomorphism, does it fail to be injective, surjective, or both?
\end{question}

Question~\ref{main question} focuses on the part of the classification problem for fiber bundles with fiber a $4$-manifold where the topological and smooth categories differ, and rephrases this in terms of the automorphism groups $\Homeo(X)$ and $\Diff(X)$.
It can be regarded as the fiber-bundle version of one of the most classical problems in differential topology:
namely, finding pairs of exotic manifolds, or finding topological manifolds that admit no smooth structure (non-smoothable topological manifolds).
Both are fundamental problems closely related to the very existence of differential topology itself.
In dimension four, it is well known that gauge theory—where one studies partial differential equations coming from physics on smooth manifolds—provides a powerful tool for resolving these problems.
What, then, is the family version or fiber-bundle version of these problems?
Let us define two fiber bundles over a common base space $B$ with common fiber $X$ to be {\it exotic as families} if they are isomorphic as topological bundles but not isomorphic as smooth bundles.
If there exists an exotic pair of families with base space the $(n+1)$-sphere, this means that the map~\eqref{eq: main comparison map on htpy grp} fails to be injective.
Similarly, we say that a topological fiber bundle with fiber a smooth manifold $X$ is a {\it non-smoothable family} if its structure group cannot be reduced from $\Homeo(X)$ to $\Diff(X)$.
The existence of a non-smoothable family over the $(n+1)$-sphere means that the map~\eqref{eq: main comparison map on htpy grp} fails to be surjective.

\subsection{Comparison with other dimensions}

An analogue of Question~\ref{main question} can of course be considered in dimensions other than four.
To explain why dimension four is of particular interest, let us compare with the situation in other dimensions.
As is well known, roughly speaking, in dimensions $ \leq 3$ there is no essential difference between the topological and smooth categories.
(For instance, any topological manifold of dimension $\leq 3$ admits a smooth structure unique up to diffeomorphism; see e.g.\ \cite{Ra1925,Moi52,HM74}.)
In contrast, in all dimensions $\geq 4$, the topological and smooth categories do differ (for example, in every dimension $\geq 4$ there exist topological manifolds that admit no smooth structure; see e.g.\ \cite[1.8.4 Corollary]{Rud16}).
Moreover, in dimensions $\leq 3$ the above slogan holds even at the level of automorphism groups, in the sense that (see e.g.\ \cite{FM12,Hat80}) for any closed oriented smooth manifold $X$ of dimension $\leq 3$, the inclusion $ \Diff(X) \inc \Homeo(X)$ is a weak homotopy equivalence.
Thus, in dimensions $\leq 3$, the answer to the analogue of the first question (``Is $\Diff(X) \inc \Homeo(X)$ a weak homotopy equivalence?'') is always ``Yes''.
By contrast, in higher dimensions there are many known examples of manifolds for which this inclusion is not a weak homotopy equivalence (for example, the sphere of one dimension lower than the dimensions $\geq 5$ in which exotic spheres exist is such a manifold. See, for instance, the first subsection of Part~1~\cite{Mil07}).
It is therefore natural to ask what happens in dimension four, where the difference between the topological and smooth categories first appears.
Until recently, however, very few examples of $4$-manifolds $X$ had been known for which the inclusion $ \Diff(X) \inc \Homeo(X)$ is not a weak homotopy equivalence.

\subsection{Gauge theory for families and its applications}

Gauge theory for families provides a tool for breaking this impasse.
In a nutshell, the idea is to consider continuous families of the partial differential equations appearing in gauge theory along a bundle of $4$-manifolds, use them to construct invariants of families, and deduce constraints coming from the fact that the structure group is $\Diff(X)$.
Gauge theory has been a tool exquisitely sensitive to the smooth structure of $4$-manifolds; by extending this framework to bundles, we obtain a means of distinguishing smooth bundles of 4-manifolds.

While applications of gauge theory to low-dimensional topology go back to Donaldson~\cite{Do83}, it is only relatively recently that gauge theory has been systematically developed for families of $4$-manifolds and applied to the study of diffeomorphism groups of $4$-manifolds.
Aside from Ruberman's pioneering results around the late 1990s~\cite{Rub98,Rub99,Rub01} and Nakamura's results soon thereafter~\cite{Naka03,Naka10}, topological applications of gauge theory for families had long remained largely unexplored.
After around 2015, many mathematicians began working in this area, leading to a remarkable series of results \cite{K16,K17,K19,B19,BK19,BK20,KT20,B20,B202,KN20,KM20,JL20,Smi20,Smi202,K21,KKN21,BK21,Ba21,B21,Smi21,LM21}
\footnote{Again, we should remark that this list only includes papers up to 2021. By now (Fall 2025), such a list has become even more extensive.}.

\subsection{Organization of the Paper}

The structure of this article is as follows.
In Section~\ref{section: Summary of the result on Diff vs. Homeo in 4D}, we summarize results answering Question~\ref{main question}—that is, results comparing the homotopy types of $\Diff(X)$ and $\Homeo(X)$.
Starting in Section~\ref{section: how to apply gauge theory to 4D top}, we explain the gauge-theoretic tools used in the proofs of these results.
Section~\ref{section: how to apply gauge theory to 4D top} reviews classical gauge theory for nonexperts as preparation for the discussion of families.
In Section~\ref{section: how to consider families gauge theory}, we describe the basic ideas of gauge theory for families.
Section~\ref{section: fin dim approx family} is the technical core of the paper, where we explain the finite-dimensional approximation of the families Seiberg--Witten equations and the resulting constraints on smooth families of $4$-manifolds.
In Section~\ref{section : whatelse} we briefly mention several additional topics.

\subsection{Notation}

We fix some notation to be used throughout this article.
For an oriented closed topological $4$-manifold $X$, we write $b^{+}(X)$ and $b^{-}(X)$ for the maximal dimensions of the positive- and negative-definite subspaces of $H^{2}(X;\R)$ with respect to the intersection form, and we write $\sigma(X)$ for the signature of $X$, i.e.\ $\sigma(X)=b^{+}(X)-b^{-}(X)$.
We write $-X$ for $X$ with reversed orientation, and $nX$ for the connected sum of $n$ copies of $X$.
The underlying differentiable manifold of a $K3$ surface (which is unique by \cite{Kod64}) is denoted by $K3$.

We also write $\Homeo(X)$ for the homeomorphism group of a topological manifold $X$ and $\Diff(X)$ for the diffeomorphism group of a smooth manifold $X$, each endowed with the $C^{0}$-topology and $C^{\infty}$-topology respectively.
(Thus the set-theoretic inclusion $\Diff(X)\inc\Homeo(X)$ is continuous but is not an inclusion of topological spaces.)
When $X$ is oriented, we write $\Homeo^{+}(X)$ and $\Diff^{+}(X)$ for the orientation-preserving homeomorphism and diffeomorphism groups, respectively.
If $X$ is an oriented closed $4$-manifold with nonzero signature, then there are no orientation-reversing self-homeomorphisms (since such a map would change the sign of the signature), so that
\[
\Homeo^{+}(X)=\Homeo(X),\quad \Diff^{+}(X)=\Diff(X).
\]
We will often write a fiber bundle with fiber $X$ over base $B$ as $X \to E \to B$.

\section{Results on $\Diff$ vs.\ $\Homeo$ in Dimension Four}
\label{section: Summary of the result on Diff vs. Homeo in 4D}

In this section, we summarize results formulated as comparisons between the diffeomorphism group and the homeomorphism group of a $4$-manifold $X$ via the natural map
\[
\pi_{n}(\Diff(X)) \to \pi_{n}(\Homeo(X)),
\]
that is, as answers to Question~\ref{main question}.
Not strictly in chronological order, up to Section~\ref{subsection: higher homotopy groups} we list results on the failure of surjectivity and injectivity of $\pi_{n}(\Diff(X)) \to \pi_{n}(\Homeo(X))$.
The results of Section~\ref{subsection: higher homotopy groups} are based on the idea of detecting non-smoothable families by means of gauge theory for families. Various further consequences of this idea are collected from Section~\ref{section: from non-smoothable to Diff Homeo} onward.

Although our main goal is to present concrete results comparing $\Diff$ and $\Homeo$, we will also mention some purely topological statements in this section.
Explanations of the gauge-theoretic tools used in the proofs are postponed to later sections; here we only indicate which tools are used.

\subsection{Classical results}
\label{section: classical results}

We begin by recalling classical results about diffeomorphism and homeomorphism groups of $4$-manifolds.
For an oriented closed (topological or smooth) $4$-manifold $X$, write
$\Aut(H^{2}(X;\Z))$ for the automorphism group of $H^{2}(X;\Z)/{\rm torsion}$ endowed with the intersection form.
(Below we will consider only the case where $H^{2}(X;\Z)$ is torsion-free--for example, when $X$ is simply-connected.)
The following result of Wall is fundamental:

\begin{thm}[Wall~\cite{Wa64}]
\label{thm: Wall}
Let $M$ be an oriented smooth simply-connected closed $4$-manifold.
Assume at least one of the following holds: (i) the intersection form of $M$ is indefinite; (ii) $b_{2}(M)\le 8$.
Then the natural map
\[
\pi_{0}(\Diff^{+}(M\# S^{2}\times S^{2})) \to \Aut(H^{2}(M\# S^{2}\times S^{2};\Z))
\]
is surjective.
\end{thm}

Outside the range where Wall's theorem applies, determining the image of $\pi_{0}(\Diff^{+}(X)) \to \Aut(H^{2}(X;\Z))$ is a difficult problem.
In the special case of the $K3$ surface, combining a result of Y.~Matsumoto \cite{Matu86} with Donaldson's theorem (\cite{Do90}, Theorem~\ref{theo: Donaldson, Morgan-Szabo}) yields a complete answer:
an element $A\in \Aut(H^{2}(K3;\Z))$ lies in the image of $\pi_{0}(\Diff^{+}(K3)) \to \Aut(H^{2}(K3;\Z))$ if and only if $A$ preserves orientation of a positive definite $3$-dimensional subspace $H^{+}(K3)\subset H^{2}(K3;\R)$.

The automorphism group of the intersection form is comparatively accessible; remarkably, in the simply-connected case it actually agrees with the topological mapping class group:

\begin{thm}[Quinn~\cite{Q86}\footnote{A gap in \cite{Q86} was recently pointed out and corrected in \cite{gabai2023pseudoisotopies}.}, Perron~\cite{P86} (both 1986)]
\label{thm: QuinnPerron}
Let $X$ be an oriented simply-connected topological closed $4$-manifold.
Then the natural map
\[
\pi_{0}(\Homeo^{+}(X)) \to \Aut(H^{2}(X;\Z))
\]
is an isomorphism.
\end{thm}

\subsection{Failure of surjectivity of $\pi_{0}(\Diff)\to \pi_{0}(\Homeo)$}
\label{section: result by gauge theory, but not gauge theory for families}

We now turn to results comparing $\Diff(X)$ and $\Homeo(X)$.
Historically, the first examples concerned the failure of surjectivity of the natural map $\pi_{0}(\Diff(X)) \to \pi_{0}(\Homeo(X))$:

\begin{thm}[Friedman--Morgan \cite{FM88,FM882}]
\label{thm: Friedman--Morgan}
Let $n>9$ be a natural number, and let $X$ be an oriented smooth $4$-manifold homeomorphic to $\CP^{2}\# n\overline{\CP}^2$.
Then the natural map
\[
\pi_{0}(\Diff(X)) \to \pi_{0}(\Homeo(X))
\]
is not surjective.
\end{thm}

\begin{thm}[Donaldson~\cite{Do90}, Morgan--Szab{\'o}~\cite{Do90}]
\label{theo: Donaldson, Morgan-Szabo}
Let $X$ be a smooth $4$-manifold homotopy equivalent to a $K3$ surface.
Then the natural map
\[
\pi_{0}(\Diff(X)) \to \pi_{0}(\Homeo(X))
\]
is not surjective.
\end{thm}

In Theorem~\ref{theo: Donaldson, Morgan-Szabo}, the case $X=K3$ is due to Donaldson \cite{Do90}; the case of a general homotopy $K3$ follows from the result of Morgan--Szab{\'o} \cite{Do90}.

The proofs of Theorems~\ref{thm: Friedman--Morgan} and \ref{theo: Donaldson, Morgan-Szabo} use ordinary (non-family) gauge theory:
one studies the action of diffeomorphisms on Donaldson or Seiberg--Witten invariants.
Along the same lines, results analogous to Theorem~\ref{theo: Donaldson, Morgan-Szabo} have also been established for many K{\"a}hler surfaces; see, for example, \cite{EbelingOkonekdiffeomorphism91,Lonnediffeomorphism98}.

Recently, however, some of these results admit alternative proofs using family-gauge-theoretic constraints due to Baraglia \cite{Ba21}.
Those alternative arguments have broad scope, allowing extensions to certain non-simply-connected closed $4$-manifolds (Nakamura--the author \cite{KN20}) and to $4$-manifolds with boundary (Taniguchi--the author \cite{KT20}).
At present, though, they apply only to $4$-manifolds with relatively small $b^{+}$ (currently $b^{+}\le 3$).

\subsection{Failure of injectivity of $\pi_{0}(\Diff)\to \pi_{0}(\Homeo)$}
\label{section: Ruberman BK}

The first application of gauge theory for families to topology is due to Ruberman \cite{Rub98}, whose work serves as a model for many of the results discussed in this article.

\begin{thm}[Ruberman~\cite{Rub98}]
\label{thm: Ruberman's first thm}
There exist closed $4$-manifolds admitting self-diffeomorphisms that are topologically isotopic to the identity but not smoothly isotopic to the identity.
Equivalently,
\begin{equation}
\label{eq: ker pi0}
\Ker\bigl(\pi_{0}(\Diff(X)) \to \pi_{0}(\Homeo(X))\bigr)
\end{equation}
is nontrivial for some closed $4$-manifolds $X$.
More concretely, $X=(2n)\CP^{2}\# m\overline{\CP}^2$ has this property for $n\ge 2$, $m\ge 10n+1$.
\end{thm}

Self-diffeomorphisms as in Theorem~\ref{thm: Ruberman's first thm}--topologically but not smoothly isotopic to the identity--are nowadays often called {\it exotic diffeomorphisms}.
Ruberman's construction of nontrivial elements in \eqref{eq: ker pi0} is ingenious and exploits the phenomenon known as {\it dissolving} of $4$-manifolds.
Very roughly, one uses the nontrivial diffeomorphism $K3\#\CP^{2}\cong 4\CP^{2}\#19\overline{\CP}^2$, typically coming from Kirby calculus, and combines it with a hand-made self-diffeomorphism on $\CP^{2}\#2\overline{\CP}^2$ (which is related to wall-crossing in Yang--Mills gauge theory).
Our construction in Theorem~\ref{thm:BK's first thm} below follows a similar idea.

Ruberman's Theorem~\ref{thm: Ruberman's first thm} is proved by considering an $S^{1}$-family of $SO(3)$ anti-self-dual Yang--Mills equations (see Section~\ref{section: how to consider families gauge theory}).
By running an analogous argument for the Seiberg--Witten equations, one can widen the scope--thanks to the cleaner structure of the wall in the Seiberg--Witten side--and obtain:

\begin{thm}[Baraglia--K.~\cite{BK20}]
\label{thm:BK's first thm}
For $X=n(K3 \# S^{2}\times S^{2})$ (with $n\ge 2$),
\[
\Ker\bigl(\pi_{0}(\Diff(X)) \to \pi_{0}(\Homeo(X))\bigr)
\]
is nontrivial.
\end{thm}

Although it is out of chronological order with Section~\ref{subsection: higher homotopy groups}, we mention here an intriguing result about Dehn twists on $4$-manifolds, based on the families Bauer--Furuta invariant (Section~\ref{subsection basic picture}).
Suppose an annulus $S^{3}\times[0,1]$ is embedded in $X$.
The Dehn twist along this annulus is the self-diffeomorphism of $X$ obtained by extending by the identity outside the annulus the map
\[
S^{3}\times[0,1]\to S^{3}\times[0,1]\ ;\ (x,t)\mapsto (\gamma(t)\cdot x,t),
\]
where $\gamma\colon[0,1]\to SO(4)$ is a loop representing the nontrivial element of $\pi_{1}(SO(4))\cong \Z/2$.

\begin{thm}[Kronheimer--Mrowka \cite{KM20}]
\label{thm: KM Dehn}
For $X=K3 \# K3$, the Dehn twist along the connected-sum neck defines a nontrivial element of
\[
\Ker\bigl(\pi_{0}(\Diff(X)) \to \pi_{0}(\Homeo(X))\bigr).
\]
\end{thm}

The proof of Theorem~\ref{thm: KM Dehn} computes the Bauer--Furuta invariant of the mapping torus family determined by the Dehn twist and shows that it is nontrivial.
No analogous argument works using families Seiberg--Witten invariants (mirroring the phenomenon that while Seiberg--Witten invariants vanish on $K3\# K3$, the Bauer--Furuta invariant does not \cite{B04}).

Developing the ideas of Theorem~\ref{thm: KM Dehn} further, J.~Lin proved:

\begin{thm}[J.~Lin \cite{JL20}]
\label{thm: Lin Dehn}
For $X=K3 \# K3 \# S^{2}\times S^{2}$, the Dehn twist along the neck between the two $K3$'s defines a nontrivial element of
\[
\Ker\bigl(\pi_{0}(\Diff(X)) \to \pi_{0}(\Homeo(X))\bigr).
\]
\end{thm}

Thus the Dehn twist considered by Kronheimer--Mrowka for $K3\# K3$ remains exotic even after taking the connected sum with $S^{2}\times S^{2}$.
Many exotic phenomena on closed $4$-manifolds are known to disappear after taking the connected sum with a single  copy of $S^{2}\times S^{2}$ (see, e.g., \cite{BS13,Akb14,DKPR,AKMR}).
Theorem~\ref{thm: Lin Dehn} shows that the property of a diffeomorphism being exotic can, in some cases, persist under a single stabilization by $S^{2}\times S^{2}$.

The proof of Theorem~\ref{thm: Lin Dehn} uses the $ \Pin(2)$-equivariant families Bauer--Furuta invariant, whereas the proof of Theorem~\ref{thm: KM Dehn} requires only the nonequivariant families Bauer--Furuta invariant.

\begin{rem}[Addendum after 2021]
\label{rem: Seifert Dehn twist}
One of the major developments after 2021 concerning exotic diffeomorphisms is the study of Dehn twists along Seifert fibered $3$-manifolds.  
Using the Seifert circle action, one can define a Dehn twist analogous to the Dehn twist along $S^{3}$ described above.  
This construction turns out to provide a rich source of exotic diffeomorphisms.
See \cite{konno-mallick-taniguchi,KLMME,KangParkTaniguchi,miyazawaDehn,KLMME2}.
As another development, it had long been an open problem whether an irreducible closed $4$-manifold can admit exotic diffeomorphisms.  
This question was answered in the affirmative by Baraglia and the author~\cite{baraglia2024irreducible4manifoldsadmitexotic}.
\end{rem}

\subsection{$\pi_{n}(\Diff)$ vs.\ $\pi_{n}(\Homeo)$ ($n>0$)}
\label{subsection: higher homotopy groups}

So far, Sections~\ref{section: result by gauge theory, but not gauge theory for families} and \ref{section: Ruberman BK} concerned $\pi_{0}(\Diff(X)) \to \pi_{0}(\Homeo(X))$.
What about higher homotopy groups?
Are there examples of $4$-manifolds $X$ for which $\pi_{n}(\Diff(X)) \to \pi_{n}(\Homeo(X))$ is not an isomorphism?
To the best of the author's knowledge, the first result in this direction is due to Watanabe, as a consequence of his work disproving the $4$-dimensional Smale conjecture.
His proof uses Kontsevich characteristic classes and is completely different from the gauge-theoretic methods that are the main theme of this article.

\begin{thm}[Watanabe~\cite{Wa18}]
\label{theo: Watanabe}
The natural map
\[
\pi_{1}(\Diff(S^{4})) \to \pi_{1}(\Homeo(S^{4}))
\]
is not injective.
\end{thm}

\begin{rem}[Addendum after 2021]
Recently, Auckly and Ruberman \cite{auckly2025familiesdiffeomorphismsembeddingspositive} generalized Ruberman's construction of exotic diffeomorphisms (Theorem~\ref{thm: Ruberman's first thm}) to higher homotopy groups:

\begin{thm}[Auckly--Ruberman~\cite{auckly2025familiesdiffeomorphismsembeddingspositive}]
\label{theo: Auckly-Ruberman}
Given $n\geq0$, there exists a simply-connected closed smooth $4$-manifold $X$ for which the natural map
\[
\pi_{n}(\Diff(X)) \to \pi_{n}(\Homeo(X))
\]
is not injective.
\end{thm}

\end{rem}

Can one also show, for higher homotopy, that $\pi_{n}(\Diff(X)) \to \pi_{n}(\Homeo(X))$ is not surjective?
The following theorem of Baraglia and the author provides the first example.

\begin{thm}[Baraglia--K.~\cite{BK21}]
\label{theo: BK K3}
The natural map
\[
\pi_{1}(\Diff(K3)) \to \pi_{1}(\Homeo(K3))
\]
is not surjective.
\end{thm}

For the proof, see Remark~\ref{rem: BK K3 proof}.
The key tool is a gauge-theoretic constraint for families of $K3$ surfaces due to Baraglia and the author \cite{BK19} (explained later as Theorem~\ref{thm: BK Steenrod}), together with the ideas used from Section~\ref{section: from non-smoothable to Diff Homeo} onward.

\subsection{Non-smoothable families with smoothable total spaces}
\label{section: from non-smoothable to Diff Homeo}

The proof of Theorem~\ref{theo: BK K3} uses the idea, mentioned in Section~\ref{sec1}, of \emph{detecting non-smoothable families}.
The following Theorem~\ref{main theorem} was the first result to detect non-smoothable families using gauge theory for families, demonstrating that families gauge theory captures a very delicate phenomenon.
Let $-E_{8}$ denote a negative-definite $E_{8}$ manifold, i.e.\ an oriented simply-connected topological closed $4$-manifold with intersection form the negative-definite $E_{8}$ lattice.

\begin{thm}[Kato--K.--Nakamura~\cite{KKN21}]
\label{main theorem}
For $3\le m\le 6$, define the topological $4$-manifold $X$ by
\begin{align}
\label{eq: Freedman decomp}
X = 2(-E_{8})\# m(S^{2}\times S^{2}).
\end{align}
(Then $X$ is homeomorphic to $K3\#(m-3)(S^{2}\times S^{2})$, so it admits a smooth structure.)
There exists a fiber bundle
\[
X \to E \to T^{m-2}
\]
with structure group $\Homeo(X)$ such that:
\begin{enumerate}
  \setlength{\parskip}{0.1cm}
  \setlength{\itemsep}{0.1cm}
\item[(1)] The total space $E$ admits a smooth structure (as a manifold).
\item[(2)] However, for \emph{no} smooth structure on $X$ can the bundle $X \to E \to T^{m-2}$ be made smooth as a fiber bundle. Equivalently, for every smooth structure on $X$, the structure group cannot be reduced from $\Homeo(X)$ to $\Diff(X)$.
\end{enumerate}
\end{thm}

\begin{rem}
Gauge theory for families is used in the proof of (2), whereas (1) is proved using Kirby--Siebenmann theory \cite{KS77}.
\end{rem}

Here is the idea behind the construction of $E$ in Theorem~\ref{main theorem}.
Choose one of the $S^{2}\times S^{2}$ summands in \eqref{eq: Freedman decomp}, consider a diffeomorphism supported there (chosen so that the eventual family-gauge-theory constraint will detect it), and extend it by the identity outside; this yields a self-homeomorphism of $X$.
Choose $(m-2)$ such summands and perform the same construction to obtain self-homeomorphisms $f_{1},\dots,f_{m-2}$ of $X$.
Their supports are disjoint, hence they commute, so we can form the (multi) mapping torus.
(Equivalently, apply the Borel construction to the continuous $\Z^{m-2}$-action on $X$ defined by $f_{1},\dots,f_{m-2}$, to obtain a family over $T^{m-2}=B(\Z^{m-2})$.)
Define $E$ to be this multi-mapping torus.
Since we used \emph{homeomorphisms} to build the mapping torus, the structure group is $\Homeo(X)$ (rather than $\Diff(X)$).
Moreover, since $-E_{8}$ does not admit a smooth structure, it is a priori unclear whether one can simultaneously smooth $f_{1},\dots,f_{m-2}$ while keeping them commuting.
That is, it is not obvious whether the structure group of $E$ can be reduced to $\Diff(X)$--and gauge theory for families shows that it cannot.
The idea of using a topological connected sum decomposition to build such examples goes back to Nakamura \cite{Naka03,Naka10}, where non-smoothability of \emph{group actions} on $4$-manifolds was studied.

\subsection{Consequences of the existence of non-smoothable families}
\label{subsection: consequences of Non-smoothable}

We record an immediate consequence of the existence of non-smoothable bundles.
For a topological bundle $E$ with fiber a smooth manifold $X$, being ``non-smoothable'' means that if $\varphi_{E}\colon B\to B\Homeo(X)$ is its classifying map, then the following lifting problem has no solution:
\begin{align*}
\xymatrix{
     & B\Diff(X)\ar[d]\\
    B \ar@{-->}[ru] \ar[r]_-{\varphi_{E}}  &  B\Homeo(X).
    }
\end{align*}
Hence, if there exists a non-smoothable bundle with fiber $X$, obstruction theory immediately implies that the inclusion $\Diff(X)\inc\Homeo(X)$ is \emph{not} a weak homotopy equivalence.

Obstruction theory further shows not only failure of weak equivalence but also, depending on the dimension of the base, up to which degree the maps $\pi_{n}(\Diff(X)) \to \pi_{n}(\Homeo(X))$ fail to be isomorphisms.
In Theorem~\ref{main theorem}, if we fix a smooth structure on $X=2(-E_{8})\# m(S^{2}\times S^{2})$ (e.g.\ regard $X$ as $K3\#(m-3)(S^{2}\times S^{2})$), then since our family has base $T^{m-2}$, we learn that $\pi_{n}(\Diff(X)) \to \pi_{n}(\Homeo(X))$ is not an isomorphism for at least one $n\in\{0,\dots,m-3\}$.
In particular, when $m=3$, for a smooth $4$-manifold $X$ homotopy equivalent to $K3$, $\pi_{0}(\Diff(X)) \to \pi_{0}(\Homeo(X))$ is not an isomorphism; more strongly, it is not surjective (since we are detecting a non-smoothable object).
This recovers Theorem~\ref{theo: Donaldson, Morgan-Szabo}.

The above results are closer to detecting elements of
\[
\Coker\bigl(\pi_{n}(\Diff(X)) \to \pi_{n}(\Homeo(X))\bigr),
\]
i.e.\ elements of $\pi_{n}(\Homeo(X))$ not coming from $\pi_{n}(\Diff(X))$.
To truly produce such cokernel elements, the obstructions up to degree $n$ must vanish (e.g.\ when $B=S^{n+1}$).
On the other hand, Wall's theorem (\cite{Wa64}, Theorem~\ref{thm: Wall}) says that for certain $4$-manifolds, $\pi_{0}(\Diff(X)) \to \pi_{0}(\Homeo(X))$ is \emph{surjective}.
By killing the low-degree obstructions via Wall's theorem, Theorem~\ref{main theorem} yields the following.
Since $\Diff(X)$ is not closed in $\Homeo(X)$ with respect to the natural $C^{0}$-topology, let us consider the homotopy quotient
\[
\Homeo(X) \sslash \Diff(X) := \bigl(E\Diff(X) \times \Homeo(X)\bigr)/\Diff(X).
\]

\begin{thm}[Kato--K.--Nakamura~\cite{KKN21}]
\label{main theorem K3S2S2 htpy quot}
For $X = K3\# S^{2}\times S^{2}$,
\[
\pi_{1}\bigl(\Homeo(X)\sslash\Diff(X)\bigr) \neq 0.
\]
\end{thm}

\begin{rem}
\label{rem: BK K3 proof}
The proof of Theorem~\ref{theo: BK K3} (non-surjectivity of $\pi_{1}(\Diff(K3)) \to \pi_{1}(\Homeo(K3))$) also proceeds by detecting a non-smoothable family over a $2$-torus $T^{2}$ and then killing low-degree obstructions.
Concretely, to detect non-smoothability we use gauge-theoretic constraints for families of $K3$ surfaces (Baraglia--the author \cite{BK19}; see Theorem~\ref{thm: BK Steenrod} below), and to kill obstructions we use results on the global Torelli theorem for $K3$ due to Giansiracusa \cite{Gian09} and Giansiracusa--Kupers--Tshishiku \cite{GKT21}.
\end{rem}

\subsection{$\Diff(X)\inc\Homeo(X)$ is not a weak homotopy equivalence for most $X^{4}$}
\label{subsection Baraglia application}

Just after \cite{KKN21}, Baraglia proved the following definitive result:

\begin{thm}[Baraglia~\cite{Ba21}]
\label{theo Baraglia}
Let $X$ be a smooth oriented simply-connected closed $4$-manifold whose intersection form is indefinite and whose signature has absolute value $>8$.
Then $\Diff(X)\inc\Homeo(X)$ is not a weak homotopy equivalence.

More precisely, there exists some $n$ with $0\le n \le \min\{b^{+}(X),b^{-}(X)\}-1$ such that the natural map
\[
\pi_{n}(\Diff(X)) \to \pi_{n}(\Homeo(X))
\]
is not an isomorphism.
\end{thm}

Thus the first part of Question~\ref{main question} (``Is $\Diff(X)\inc\Homeo(X)$ a weak homotopy equivalence?'') is answered in the negative for most simply-connected closed $4$-manifolds.

\begin{rem}[Addendum after 2021]
In 2023, Lin--Xie~\cite{lin2023configuration} proved that for any compact, orientable smooth $4$-manifold $X$, the map $\Diff(X) \inc \Homeo(X)$ is not rationally $2$-connected (and hence not a weak homotopy equivalence). 
Their technique is based on configuration space integrals, as in Watanabe's work~\cite{Wa18}, rather than on gauge theory.
\end{rem}

\begin{rem}
\label{rem: Baraglia spin}
When $X$ is spin, the range in Theorem~\ref{theo Baraglia} improves: there exists $n$ with $0\le n \le \min\{b^{+}(X),b^{-}(X)\}-3$ such that $\pi_{n}(\Diff(X)) \to \pi_{n}(\Homeo(X))$ is not an isomorphism.
For the gauge-theoretic reason, see Remark~\ref{rem: Baraglia constraint other cohomologies}.
\end{rem}

The proof of Theorem~\ref{theo Baraglia} proceeds as follows.
First, one proves a family version of Donaldson's diagonalization theorem (explained later as Theorem~\ref{thm: Baraglia constraint}), yielding constraints on smooth families of $4$-manifolds.
As in the construction of the non-smoothable families in Theorem~\ref{main theorem} (Section~\ref{section: from non-smoothable to Diff Homeo}), one builds a continuous bundle of $X$ over a torus of dimension $b^{+}(X)$ as a multi-mapping torus.
Using the prepared constraint, one shows that this bundle is non-smoothable.

Although Theorem~\ref{theo Baraglia} is stated for simply-connected $4$-manifolds, the same argument applies, for example, to manifolds obtained by taking a connected sum of a simply-connected $4$-manifold with finitely many copies of $S^{1}\times S^{3}$.
On the other hand, if one takes connected sums with factors such as $S^{1}\times Y$ or $S^{2}\times \Sigma_{g}$ (where $Y$ is any oriented closed $3$-manifold and $\Sigma_{g}$ is a closed oriented surface of genus $g\ge 1$), the proof of Theorem~\ref{theo Baraglia} no longer applies.
For this type of non-simply-connected $4$-manifold, the $\mathrm{Pin^{-}}(2)$-monopole theory developed by Nakamura \cite{Naka13,Naka15} works effectively.
In joint work \cite{KN20}, Nakamura and the author proved results parallel to Theorem~\ref{theo Baraglia} for such manifolds, using a family version of the $\mathrm{Pin^{-}}(2)$ monopole equations.

\subsection{$\Diff$ vs.\ $\Homeo$ for $4$-manifolds with boundary}
\label{section: boundary}

Baraglia's Theorem~\ref{theo Baraglia} is a very powerful theorem for closed 4-manifolds.
Taniguchi and the author \cite{KT20} extended it to $4$-manifolds with boundary.
To state the result, we briefly recall the Fr{\o}yshov invariant for $3$-manifolds (see Section~\ref{subsection KT constraint} for details).
This is a particularly important numerical invariant defined via Floer theory for $3$-manifolds.
For an oriented $\,\spc$ rational homology $3$-sphere $(Y,\frakt)$, the Fr{\o}yshov invariant $\delta(Y,\frakt)$ takes values in $\frac{1}{8}\Z$.
Although $\delta(Y,\frakt)$ depends on the spin$^{c}$ structure $\frakt$, when $Y$ is an integral homology sphere the spin$^{c}$ structure is unique, and we suppress it from the notation, writing simply $\delta(Y)$.
For integral homology spheres $Y$, the invariant $\delta$ takes values in $\Z$.
The group $\Theta_{3}^{\Z}$ of integral homology $3$-spheres modulo homology cobordism is a central object in low-dimensional topology, and the Fr{\o}yshov invariant induces a surjective homomorphism $\delta\colon \Theta_{3}^{\Z}\to \Z$.

For a manifold with boundary $X$, set $\Homeo(X,\partial)=\{f\in \Homeo(X)\mid f|_{Y}=\id_{Y}\}$ and define $\Diff(X,\partial)$ similarly.
Baraglia's theorem (Theorem~\ref{theo Baraglia}) extends to the following statement for 4-manifolds with boundary:

\begin{thm}[K.--Taniguchi~\cite{KT20}]
\label{theo KT appl}
Let $X$ be a smooth oriented simply-connected $4$-manifold with boundary, with $\sigma(X)\le 0$ and indefinite intersection form.
Assume the boundary $\partial X=Y$ is connected and an integral homology $3$-sphere.
Suppose either $\sigma(X)<-8$ and $\delta(Y)\le 0$, or $X$ is spin and $-\sigma(X)/8>\delta(Y)$.
Then
\begin{align*}
\Diff(X,\partial) \inc \Homeo(X,\partial)\quad \text{and}\quad
\Diff(X) \inc \Homeo(X)
\end{align*}
are not weak homotopy equivalences.
More precisely, the natural map $\pi_{n}(\Diff(X,\partial)) \to \pi_{n}(\Homeo(X,\partial))$ fails to be an isomorphism for some $n\in\{0,\ldots,b^{+}(X)-1\}$, and $\pi_{n}(\Diff(X)) \to \pi_{n}(\Homeo(X))$ fails to be an isomorphism for some $n\in\{0,\ldots,b^{+}(X)\}$.
\end{thm}

For instance, the Gompf nucleus $N(2n)$ inside the elliptic surface $E(2n)$ is an example that satisfies the assumption of Theorem~\ref{theo KT}.
The proof of Theorem~\ref{theo KT appl} is based on Theorem~\ref{theo KT}, which will be explained later.

\begin{rem}[Addendum after 2021]
For $4$-manifolds with boundary, many exotic diffeomorphisms have been found since 2021.  
Many of these are Dehn twists along Seifert fibered $3$-manifolds, as mentioned in Remark~\ref{rem: Seifert Dehn twist}.  
For another example, see \cite{IKMT25}.
\end{rem}

\section{How to Apply Gauge Theory to 4D Topology}
\label{section: how to apply gauge theory to 4D top}

From here on, aiming toward an outline of proofs for some of the results stated in Section~\ref{section: Summary of the result on Diff vs. Homeo in 4D}, we begin explaining the content of gauge theory.
In this section, as preparation for gauge theory \emph{for families}, we describe---with as little jargon as possible---the basic methodology behind the classical applications of gauge theory to $4$-dimensional topology.
We ignore technical hypotheses in the exposition; for precise and detailed accounts, see \cite{DK90,FU91,Mo96,Ni20}.

\subsection{Invariants and constraints}
The standard recipe for applying gauge theory to $4$-dimensional topology can be summarized in the following three steps.

\begin{description}
\item[Step 1] On an oriented smooth closed $4$-manifold, write down a system of partial differential equations: either the anti-self-dual Yang--Mills equations or the Seiberg--Witten equations.
(Explicitly,
\begin{align*}
F_A^{+}=0\qquad\text{and}\qquad
\left\{
\begin{array}{l}
F_A^{+}=\sigma(\Phi,\Phi),\\[6pt]
\slashed{D}_{A}\Phi=0
\end{array}
\right.
\label{SW eq.}
\end{align*}
respectively, but understanding the symbols is not necessary for what follows.)
To formulate these PDEs one needs, of course, a smooth structure on the manifold.
More precisely, in addition to the smooth structure one fixes some non-topological data such as a Riemannian metric, together with certain topological data (e.g.\ a principal $G$-bundle for a suitable Lie group $G$, or a spin$^{c}$ structure).

\item[Step 2] From these are nonlinear PDEs, one extracts a finite-dimensional space called the moduli space, via the following procedure.
The solution space of the PDE (for generic choices) is an infinite-dimensional manifold endowed with an action of an infinite-dimensional group, called the gauge group.
Taking the quotient of the solution space by this action produces the moduli space, which, under genericity assumptions, is a finite-dimensional manifold
(although depending on the topology of the $4$-manifold, quotient singularities may occur).

\item[Step 3] Extract information about the original $4$-manifold from the moduli space.
There are two broad approaches:
\begin{itemize}
\item Method 1: Use the moduli space to define differential-topological \emph{invariants} of the $4$-manifold and distinguish manifolds by those invariants (one must then prove that the resulting invariant is independent of the choices of non-topological auxiliary data made in Step 1).
Typically one reduces to the case where the moduli space is $0$-dimensional and obtains an integer-valued invariant by counting points.
\item Method 2: Deduce \emph{constraints} on classical invariants of the $4$-manifold (typically on the intersection form) from geometric properties of the moduli space.
\end{itemize}
\end{description}

\begin{table}[htb]
\caption{Invariants and constraints}
\begin{center}
  \begin{tabular}{|c|c|}\hline
    Method 1: invariants & Method 2: constraints \\ \hline
    Donaldson invariants~\cite{Do90} & Donaldson's diagonalization theorem~\cite{Do83} \\
    Seiberg--Witten invariants~\cite{W94} & Donaldson's Theorems B, C~\cite{Do86} \\
    Bauer--Furuta invariant~\cite{BF04} & Furuta's $10/8$-inequality~\cite{Fu01} \\ \hline
  \end{tabular}
  \end{center}
\label{label: table invariants and constraints}
\end{table}

Table~\ref{label: table invariants and constraints} lists typical examples of invariants (Method 1) and constraints (Method 2).
Strictly speaking, the bottom row---the Bauer--Furuta invariant and Furuta's $10/8$-inequality---uses information more refined than the moduli space itself: one views the Seiberg--Witten equations as a map between infinite-dimensional spaces and extracts information from a \emph{finite-dimensional approximation} of that map.
We will explain this in Section~\ref{section: fin dim approx of Seiberg--Witten eq}.
First, we give a more detailed, nontechnical picture of moduli spaces via a finite-dimensional model.

\subsection{Finite-dimensional models}
\label{subsection Kuranishi model}

The anti-self-dual Yang--Mills equations and the Seiberg--Witten equations (taking the gauge action into account) are nonlinear elliptic PDEs.
Ellipticity entails Fredholm properties; intuitively, it means such PDEs admit finite-dimensional models in the following sense.
One can view the equations as the zero set of a section
\begin{equation}
\label{main section of a Hilbert bundle}
s:\ \mathcal{B}\to \mathcal{E}
\end{equation}
of a Hilbert bundle $\mathcal{H}\to \mathcal{E}\to \mathcal{B}$, where $\mathcal{B}$ is an infinite-dimensional manifold and each fiber is a Hilbert space $\mathcal{H}$.
The zero-locus $s^{-1}(0)$ is the moduli space (solutions modulo gauge).
If at a zero $x\in s^{-1}(0)$ the derivative $ds_{x}:T_{x}\mathcal{B}\to \mathcal{H}$ is surjective (this can be arranged generically), then by the implicit function theorem in infinite dimensions, $s^{-1}(0)$ is a manifold; Fredholmness guarantees moreover that it is \emph{finite}-dimensional.
In fact, more strongly, near each point it is modeled by the zero set of a section
\[
s':\ B\to E
\]
of a finite-rank vector bundle $E\to B$ over a finite-dimensional manifold $B$; this is the \emph{Kuranishi model}.
Heuristically, although both the base $\mathcal{B}$ and the fiber $\mathcal{H}$ are infinite-dimensional, the differential of $s$ identifies the infinite-dimensional parts, leaving a finite-dimensional residual part captured by $s':B\to E$.
In particular, $\dim s^{-1}(0)=\dim B-\operatorname{rank}E$, the \emph{formal} or \emph{virtual} dimension, computable via the Atiyah--Singer index theorem \cite{AS68}.
As suggested by the finite-dimensional model, this formal dimension is defined independently of whether $s$ is transverse to the zero section and is also independent of non-topological auxiliary choices, such as the Riemannian metric.

\subsection{Finite-dimensional approximation of the Seiberg--Witten equations and gauge-theoretic constraints}
\label{section: fin dim approx of Seiberg--Witten eq}

We now explain the finite-dimensional approximation introduced by Furuta in his proof of the $10/8$-inequality \cite{Fu01} and subsequently developed into the Bauer--Furuta invariant \cite{BF04}.
(References for this subsection and the next, Section~\ref{subsection BF inv}, include Furuta's original paper \cite{Fu01}, the Bauer--Furuta paper \cite{BF04}, and Bauer's gluing paper \cite{B04}; see also Furuta's preprint \cite{FurutaBonn}, his survey \cite{F02}, and Bauer's survey \cite{Ba04}.)
Finite-dimensional approximation will also be fundamental later in the family setting.
Moreover, it provides a unified explanation of the various gauge-theoretic constraints listed on the right side of Table~\ref{label: table invariants and constraints}.
Recall two of those statements (we discuss Donaldson's Theorems B, C later):

\begin{thm}[Donaldson~\cite{Do83}]
\label{thm: Donaldson diag}
Let $X$ be a smooth oriented closed $4$-manifold with negative-definite intersection form.
Then the intersection form of $X$ is diagonalizable over $\Z$.
\end{thm}

\begin{thm}[Furuta~\cite{Fu01}]
\label{thm: Furuta 10/8}
Let $X$ be a smooth oriented spin closed $4$-manifold with indefinite intersection form.
Then
\[
b_{2}(X)\ \ge\ \frac{5}{4}\,|\sigma(X)|+2.
\]
\end{thm}

Compared with Freedman's results for topological $4$-manifolds \cite{Fre82}, these theorems impose extremely strong restrictions on the intersection forms of \emph{smooth} $4$-manifolds.
What we wish to emphasize is that, despite their very different forms (and original proofs), both theorems can be derived in a parallel way from finite-dimensional approximations of the Seiberg--Witten equations; the difference lies only in which cohomology theory is applied.

In the Seiberg--Witten case, one can choose the finite-dimensional model large enough to capture the entire moduli space $s^{-1}(0)$.  
The key point that makes this possible is the crucial fact that the Seiberg--Witten moduli space is compact.
After suitable rephrasing, this yields a continuous map between finite-dimensional spheres
\[
f:\ S^{m}\to S^{n}\qquad (m,n\gg 0).
\]
Here the domain and target spheres arise as one-point compactifications of finite-dimensional approximations of the function spaces naturally associated with the PDE.
Once we have the finite-dimensional map $f$, we can study not only the zero set $f^{-1}(0)$ (which corresponds to the moduli space, modulo a natural $S^{1}$-action), but the map $f$ itself from a homotopy-theoretic viewpoint.

Crucially, there is a natural action of a Lie group $G$ on the domain and target spheres, and the map $f$ is \emph{equivariant}, reflecting an internal symmetry of the Seiberg--Witten equations.
(Generally $G=S^{1}$; when the $4$-manifold is spin, $G=\Pin(2)$.)
More concretely, there are real $G$-representations $V_{\R},W_{\R}$ and complex representations $V_{\C},W_{\C}$ such that
\[
S^{m}=V_{\R}^{+}\wedge V_{\C}^{+},\qquad
S^{n}=W_{\R}^{+}\wedge W_{\C}^{+},
\]
where ${}^{+}$ denotes one-point compactification, and $G$ acts linearly but differently on the real and complex parts.
(For $G=S^{1}$, $V_{\R},W_{\R}$ carry the trivial representation, while $V_{\C},W_{\C}$ carry the scalar action.)
These two representation types correspond to the two linearized operators appearing in the Seiberg--Witten setup: the real-linear Atiyah--Hitchin--Singer operator and the complex-linear Dirac operator.
The differences
\[
\dim V_{\R}-\dim W_{\R},\qquad \dim V_{\C}-\dim W_{\C}
\]
are computed by index theory and give rise to two kinds of characteristic numbers.
From the existence of an equivariant map $f$, Borsuk--Ulam-type theorems in appropriate equivariant cohomology theories produce inequalities relating these characteristic numbers, which in turn yield strong restrictions on smooth $4$-manifolds.

In practice, applying different Borsuk--Ulam-type theorems gives:
Furuta's $10/8$-inequality (Theorem~\ref{thm: Furuta 10/8}) arises by applying a Borsuk--Ulam statement in $\Pin(2)$-equivariant $K$-theory $K_{\Pin(2)}$ to $f$.
Similarly, applying Borsuk--Ulam statements in $S^{1}$-equivariant ordinary cohomology $H^{*}_{S^{1}}$ and in $\Pin(2)$-equivariant ordinary cohomology $H^{*}_{\Pin(2)}$ to $f$ recovers, respectively, Donaldson's diagonalization theorem (\cite{BF04}, Theorem~\ref{thm: Donaldson diag}) and Donaldson's Theorems B, C (see the right column of Table~\ref{label: table invariants and constraints}).
In short, different equivariant cohomology theories---$H^{*}_{S^{1}}$, $H^{*}_{\Pin(2)}$, $K_{\Pin(2)}$---produce distinct Borsuk--Ulam principles, which in turn yield different constraints on smooth $4$-manifolds.

\begin{rem}
Donaldson's Theorems B, C give constraints on intersection forms of smooth spin closed $4$-manifolds with small $b^{+}$.
As statements for \emph{closed} $4$-manifolds they are subsumed by the $10/8$-inequality, but they extend to manifolds with boundary \cite{FLin17} due to F.~Lin, yielding constraints different from Manolescu's $10/8$-type inequality for manifolds with boundary \cite{Ma14}.
In the family setting they again give constraints distinct from the family $10/8$-inequality \cite{Ba21}.
Note that F.~Lin's proof \cite{FLin17} does not use finite-dimensional approximation; it is based on his $\Pin(2)$-monopole Floer homology \cite{FLin18}.
A proof via finite-dimensional approximation (extended to the relative/family case) is given by Taniguchi and the author \cite{KT20}, which can be viewed as a relative version of Baraglia's argument \cite{Ba21}.
\end{rem}

\subsection{The Bauer--Furuta invariant}
\label{subsection BF inv}
Beyond constraints, the finite-dimensional approximation also leads to an invariant of $4$-manifolds: the Bauer--Furuta invariant \cite{BF04}.
It is defined from the finite-dimensional approximation $f:S^{m}\to S^{n}$ of the Seiberg--Witten equations by stabilizing to absorb ambiguities in the construction; the result is an element of a stable cohomotopy group.
The Seiberg--Witten invariant corresponds to counting $f^{-1}(0)$ (more precisely, the zero set modulo the natural $S^{1}$-action), i.e.\ to a degree defined using ordinary cohomology on the quotient of the domain by the $S^{1}$-action, and the Bauer--Furuta invariant naturally recovers the Seiberg--Witten invariant.
Moreover, the Bauer--Furuta invariant is strictly stronger than the Seiberg--Witten invariant (see, e.g., \cite{B04,FKM01}); among invariants defined from the Seiberg--Witten equations for closed $4$-manifolds, it is currently the most informative.

\subsection{The Seiberg--Witten Floer stable homotopy type}
\label{subsection Seiberg--Witten Floer stable homotopy type}

We now touch on $3$-manifolds and $4$-manifolds with boundary.
Manolescu \cite{Man03} considered a $3$-dimensional analogue of Furuta's and Bauer--Furuta's finite-dimensional approximation for the Seiberg--Witten equations and, applying Conley index theory, constructed the \emph{Seiberg--Witten Floer stable homotopy type}.
A landmark application was his disproof of the triangulation conjecture, one of the major open problems in topology \cite{Man16}.
The Seiberg--Witten Floer stable homotopy type is a space-level refinement of monopole Floer homology \cite{KM07} for $3$-manifolds:
monopole Floer homology is, roughly, a gauge-theoretic invariant assigning an abelian group to a $3$-manifold, constructed as the Morse homology of a functional associated with the Seiberg--Witten equations on an infinite-dimensional manifold.
The Seiberg--Witten Floer \emph{stable homotopy type} is a space (more precisely, a stable homotopy type or a spectrum) whose (equivariant) singular homology recovers monopole Floer homology (this recovery is proved in \cite{LM18}).
Monopole Floer homology contains powerful information about $3$-manifolds and cobordisms between them, and having a space-level object allows one to apply various generalized cohomology theories to extract information beyond Floer homology (recall how Furuta used $K$-theory in the closed $4$-dimensional setting to obtain the $10/8$-inequality).
This idea goes back to \cite{CJS95}, and there have been various attempts to construct space-level refinements for other Floer theories.
Analytic constructions are, to date, realized essentially only in Manolescu's work and its generalizations \cite{KM02,KLS18,SS21}.
(For combinatorial space-level refinements in neighboring areas such as Khovanov homology, Bar--Natan homology, and knot Floer homology, see \cite{LS14,S21,MS21}; for partial lifts to generalized cohomology in symplectic Floer theory, see \cite{AB21}.)

The Seiberg--Witten Floer stable homotopy type provides the target for the \emph{relative} (i.e.\ with boundary) Bauer--Furuta invariant of $4$-manifolds.
This is the space-level counterpart of the fact that the relative Seiberg--Witten invariant takes values in monopole Floer homology.

\section{Basic Idea of Gauge Theory for Families}
\label{section: how to consider families gauge theory}

We can now finally begin the explanation of gauge theory for families.
Rephrasing the setup: on an oriented smooth closed $4$-manifold $X$ we equip various auxiliary data (a Riemannian metric, a principal $G$-bundle or a $\spc$ structure) and consider either the anti-self-dual equations or the Seiberg--Witten equations.
This gives a section $s:\mathcal{B}\to\mathcal{E}$ of a bundle $\mathcal{H}\to\mathcal{E}\to\mathcal{B}$ with Hilbert-space fiber $\mathcal{H}$ over an infinite-dimensional manifold $\mathcal{B}$.
The zero set $s^{-1}(0)$ is locally modeled as the zero set of a section $s':B\to E$ of a finite-rank vector bundle over a finite-dimensional manifold.
Assume now that the formal dimension is negative.
In the finite-dimensional model this corresponds to the fiber dimension being larger than the base dimension.
Since in applications we typically study properties invariant under perturbations, we lose no generality in assuming a generic situation where $s:\mathcal{B}\to\mathcal{E}$ is transverse to the zero section.
But with negative formal dimension, transversality implies $s^{-1}(0)=\emptyset$.
An empty moduli space yields no information, and Furuta's finite-dimensional approximation method is also unavailable: the finite-dimensional approximation of $s$ becomes homotopic to a constant map.
Thus ordinary gauge theory is powerless when the formal dimension is negative.

On the other hand, in this situation one can sometimes extract meaningful information by considering \emph{families}.
Ruberman \cite{Rub98} was the first to apply this observation to topology: by implementing the idea below for $4$-manifold bundles with base $S^{1}$, he proved Theorem~\ref{thm: Ruberman's first thm}.
Let $B$ be a finite-dimensional manifold, and suppose we are given a smooth family of $4$-manifolds $X$ parametrized by $B$, i.e.\ a smooth fiber bundle $X\to E\to B$.
Assume moreover that the auxiliary data needed for gauge theory vary continuously along the fibers over $B$ (e.g.\ for the $SU(2)$-Yang--Mills equations, a family of $SU(2)$-bundles $P\to X$; for the Seiberg--Witten equations, a family of $\spc$ structures).
Fix a family of auxiliary choices along $E$ (for instance, a family of Riemannian metrics).
Then the gauge-theoretic equations become \emph{parametrized over $B$}: we obtain a family of infinite-dimensional bundles and sections as in \eqref{main section of a Hilbert bundle},
\[
\qquad
s=\bigsqcup_{b\in B}s_{b} \;:\; \bigsqcup_{b\in B}\mathcal{B}_{b}\;\to\; \bigsqcup_{b\in B}\mathcal{E}_{b}.
\]
Consider the zero set of this parametrized section, $s^{-1}(0)=\bigsqcup_{b\in B}s_{b}^{-1}(0)$, called the {\it parameterized moduli space}.
Its dimension equals the sum of the original formal dimension and $\dim B$.
For example, if $\dim B$ is the negative of the formal dimension, then generically $s^{-1}(0)$ is a $0$-manifold, which may be nonempty.
(Equivalently, in the finite-dimensional model, imagine a family of sections $s':B\to E$ of finite-rank bundles over a finite-dimensional manifold varying over a parameter space $B$.)

Thus the basic idea of gauge theory for families is to use parameterized moduli spaces precisely in situations where the formal dimension is negative.
As the finite-dimensional model also suggests, if the family $X\to E\to B$ is a \emph{trivial} bundle, then the parameterized moduli space $s^{-1}(0)$ is still empty, yielding no information.
Conversely, if, after perturbing, one can count the parameterized moduli space and show it never becomes empty, then the family $X\to E\to B$ must be a nontrivial bundle.

Ruberman's theorem (Theorem~\ref{thm: Ruberman's first thm}), which first detected exotic diffeomorphisms in dimension four, is proved by carrying out this idea for an $S^{1}$-family of $SO(3)$-anti-self-dual Yang--Mills equations.
One can do the same with the Seiberg--Witten equations, which enlarges the scope; this yields Theorem~\ref{thm:BK's first thm}.
Both Theorem~\ref{thm: Ruberman's first thm} and Theorem~\ref{thm:BK's first thm} define $\Z$- or $\Z/2$-valued invariants by counting the moduli space over $S^{1}$ in the case of formal dimension $-1$, and then prove their nontriviality by analytic arguments based on wall-crossing and gluing.

Moreover, the constructions used in Theorems~\ref{thm: Ruberman's first thm} and \ref{thm:BK's first thm} admit far-reaching generalizations.
For instance, Li--Liu \cite{LiLiu01} define $\Z$- or $\Z/2$-valued numerical invariants for families of $\spc$ $4$-manifolds over closed base manifolds via the Seiberg--Witten equations.
More generally, for any $4$-manifold bundle whose structure group is a suitable subgroup of the diffeomorphism group (e.g.\ preserving the isomorphism type of the principal $SO(3)$-bundle or of the $\spc$ structure), one can define characteristic classes packaging the information from the $SO(3)$-anti-self-dual Yang--Mills or Seiberg--Witten equations (the author \cite{K21}).
The classes of \cite{K21} are defined for arbitrary base spaces; when the base is a closed manifold, pairing with the fundamental class recovers the numerical invariants of \cite{LiLiu01}.
A family version of the Bauer--Furuta invariant can likewise be defined.
Constructions appear already in Bauer--Furuta \cite[Theorem~2.6]{BF04} and in Furuta's preprint \cite{FurutaBonn}, and were later reformulated by Szymik \cite{Szy10}.
The families Bauer--Furuta invariant recovers the families Seiberg--Witten invariants (Baraglia--the author \cite{BK19}).

\section{Finite-dimensional approximation of the Seiberg--Witten equations for families}
\label{section: fin dim approx family}

Many of the results in Section~\ref{section: Summary of the result on Diff vs. Homeo in 4D}—specifically Theorem~\ref{main theorem}, Theorem~\ref{main theorem K3S2S2 htpy quot}, Theorem~\ref{theo: BK K3}, Theorem~\ref{theo Baraglia}, Theorem~\ref{theo KT appl}, Theorem~\ref{thm: KM Dehn}, and Theorem~\ref{thm: Lin Dehn}—are all based on the finite-dimensional approximation of the Seiberg--Witten equations for families. In this section, we explain the basic idea of the method and present examples of the concrete constraints it yields for families of $4$-manifolds.

\subsection{Basic picture}
\label{subsection basic picture}

Let us set up the notation. Let $B$ be a finite CW complex (for most practical purposes, one may take $B$ to be a compact manifold). Let $X$ be an oriented smooth closed 4-manifold, and let $X \to E \to B$ be an oriented smooth fiber bundle with fiber $X$. By an oriented smooth fiber bundle we mean that the structure group reduces to the group $\Diff^{+}(X)$ of orientation-preserving diffeomorphisms of $X$.

\begin{rem}
\label{rem: smoothing by MW09}
If $B$ is a smooth manifold, one may replace $E$ by an isomorphic bundle so that the total space $E$ becomes a smooth manifold and the projection $E \to B$ is a smooth map~\cite{MW09}. For simplicity, one often assumes from the outset that $B$ is smooth, but at present there seems to be no situation in families gauge theory where the smoothness of the base $B$ is essentially needed. This is analogous to the theory of indices of families of linear elliptic operators~\cite{AS71}, where smoothness of the base is not required.
\end{rem}

Next, we fix additional data along $E$ to write down the Seiberg--Witten equations. Assume that $X$ is endowed with a $\spc$ structure $\fraks$, and that $E$ carries along each fiber a continuous family of copies of $\fraks$. Precisely, the structure group of $E$ reduces to the automorphism group of the $\spc$ 4-manifold $(X,\fraks)$. We say that we are given a smooth fiber bundle of $\spc$ 4-manifolds $(X,\fraks) \to E \to B$, and we use this notation%
\footnote{To be more precise: instead of defining spin and $\spc$ structures via the double cover of $SO(n)$, we use the double cover of the Lie group $GL^{+}(n,\R)$ of real $n\times n$ matrices with positive determinant. This avoids building a Riemannian metric into the definition, which is convenient since our structure group is all of $\Diff(X)$ rather than the isometry group for a fixed metric.}.
In addition, we choose a fiberwise Riemannian metric on $E$, i.e. a family of Riemannian metrics along the fibers varying continuously over $B$.

With this in hand, each fiber of $E$ carries a Seiberg--Witten equation depending continuously on the parameter in $B$. Using the compactness of $B$, we can carry out the finite-dimensional approximation of the Seiberg--Witten equations (as explained in Section~\ref{section: fin dim approx of Seiberg--Witten eq}) simultaneously over $B$. The outcome is finite-rank real vector bundles $V_{\R}\to B$, $W_{\R}\to B$, finite-rank complex vector bundles $V_{\C}\to B$, $W_{\C}\to B$, and a fiber-preserving continuous map between their Thom spaces
\begin{equation}
\label{eq: fin. dim. approx. family}
\begin{split}
\xymatrix{
f : \Th(V_{\R}\oplus V_{\C}) \ar@{->}[d] \ar[r] & \Th(W_{\R}\oplus W_{\C}) \ar@{->}[d] \\
B \ar@{=}[r] & B
}
\end{split}
\end{equation}
Exactly as in the unparametrized case of Section~\ref{section: fin dim approx of Seiberg--Witten eq}, there is a Lie group $G$ acting fiberwise on $V_{\R},W_{\R}$ and on $V_{\C},W_{\C}$. Concretely, $G=S^{1}$ for a general $\spc$ structure $\fraks$, while $G=\Pin(2)$ when $\fraks$ is a spin structure. For instance, when $G=S^{1}$ the action on $V_{\R},W_{\R}$ is trivial and the action on $V_{\C},W_{\C}$ is induced by scalar multiplication on the fibers. We let $G$ act trivially on the base $B$, so these are $G$-equivariant bundles over $B$, and the map $f$ between their Thom spaces is $G$-equivariant.

In Section~\ref{section: fin dim approx of Seiberg--Witten eq} we saw that the existence of a $G$-equivariant map $f$ between spheres yields, via Borsuk-Ulam type theorems, inequalities among characteristic numbers of the 4-manifold—recovering results such as Donaldson's diagonalization theorem and the $10/8$-inequality. The families version of this story says: from the existence of a $G$-equivariant map $f$ between Thom spaces of certain vector bundles naturally associated to the family, one obtains constraints on the characteristic classes of those bundles (often in combination with the unparametrized characteristic numbers).

As in Subsection~\ref{subsection BF inv}, one can absorb the choices in the finite-dimensional approximation by passing to stable homotopy, and thereby define a \emph{families Bauer--Furuta invariant}. In actual computations, however, one frequently needs to fix trivializations of the objects appearing in the approximation; when uniqueness of such trivializations fails, this becomes a serious obstruction. This issue does not occur in the unparametrized setting.

\subsection{Constraints obtained from the finite-dimensional approximation for families}
\label{section: fin dim approx family constraint}

Let us state a concrete constraint obtained from the finite-dimensional approximation for families. As an example, we take Baraglia's families version of Donaldson's diagonalization theorem (Theorem~\ref{thm: Baraglia constraint}), which underlies the proof of Theorem~\ref{theo Baraglia}. Keep $B$, $(X,\fraks)$, and the bundle $(X,\fraks)\to E\to B$ as in Subsection~\ref{subsection basic picture}: $B$ is a finite CW complex, $X$ is an oriented smooth closed 4-manifold, $\fraks$ is a $\spc$ structure on $X$, and $E$ is a smooth bundle of $\spc$ 4-manifolds $(X,\fraks)$ over $B$. Let $\fraks_{E}$ be a fiberwise $\spc$ structure on $E$ that restricts to $\fraks$ on each fiber.

The key elementary invariant of $E$ in gauge theory for families is the following vector bundle.
Associated to the bundle $X\to E\to B$ there is, uniquely up to isomorphism, a real vector bundle
\[
\R^{b^{+}(X)} \to H^{+}(E) \to B .
\]
This is independent of the $\spc$ structure $\fraks$ and depends only on the oriented \emph{topological} bundle structure of $E$, i.e. on its structure as a $\Homeo^{+}(X)$-bundle; we do not use any reduction to $\Diff^{+}(X)$. Intuitively, over each $b\in B$ the fiber encodes a maximal positive-definite subspace of $H^{2}(E_{b};\R)$ with respect to the intersection form. 
The precise definition of the vector bundle $H^{+}(E)$ is as follows.
Let $\Gr^{+}(H^{2}(X;\R))$ be the space of $b^{+}(X)$-dimensional subspaces of $H^{2}(X;\R)$ that are  positive-definite for the intersection form. This ``Grassmannian'' is contractible (see, e.g., \cite{LiLiu01}). The group $\Homeo^{+}(X)$ acts naturally on $\Gr^{+}(H^{2}(X;\R))$, hence the $\Homeo^{+}(X)$-bundle structure on $E$ induces a bundle
\[
\Gr^{+}(H^{2}(X;\R)) \to \Gr^{+}_{E} \to B
\]
with fiber $\Gr^{+}(H^{2}(X;\R))$.
Since the fiber $\Gr^{+}(H^{2}(X;\R))$ is contractible, the bundle $\Gr^{+}_{E}$ admits a section unique up to homotopy. Choosing such a section determines a vector bundle $\R^{b^{+}(X)} \to H^{+}(E) \to B$, and the uniqueness up to homotopy implies that its isomorphism class is canonically determined by $E$.

In general there is no canonical choice of $H^{+}(E)$. If the structure group of $E$ reduces to $\Diff^{+}(X)$, then by choosing a fiberwise family of Riemannian metrics one can assemble the spaces of self-dual harmonic 2-forms fiberwise into a vector bundle, which is one realization of $H^{+}(E)$. For applications to non-smoothable families one must work in the topological category, so it is important to have the definition using only the $\Homeo^{+}(X)$-bundle structure as above.

Next, using the reduction to $\Diff^{+}(X)$ and the fiberwise $\spc$ structure $\fraks_{E}$, we define another (virtual) bundle. Choose a fiberwise family of Riemannian metrics on $E$. Using $\fraks_{E}$, consider along the fibers the family of determinant line bundles and the bundle over $B$ whose fiber on $b \in B$ consists of all connections on the line bundle on $E_{b}$ (this fiber is contractible—indeed an infinite-dimensional affine space modeled on $\Omega^{1}(X)$). Choosing a section is the same as choosing a family of connections $\{A_{b}\}_{b\in B}$. We can then form the family of $\spc$ Dirac operators $\{\slashed{D}_{A_{b}}\}_{b\in B}$. Its index
\[
\ind \slashed{D}_{E} := \ind \{\slashed{D}_{A_{b}}\}_{b\in B} \in K(B)
\]
is independent of all choices, since the ambiguities are contractible.

Thus from $E$ we obtain the real vector bundle $H^{+}(E)$ and the complex (virtual) vector bundle $\ind \slashed{D}_{E}$. These correspond to the two linearized pieces appearing in the Seiberg--Witten equations: the real Atiyah-Hitchin-Singer operator and the complex Dirac operator. They may be viewed as the ``linearization'' of the bundle $E$. Using them, one can describe the bundles appearing in the families finite-dimensional approximation \eqref{eq: fin. dim. approx. family}: the real bundles $V_{\R},W_{\R}$ and the complex bundles $V_{\C},W_{\C}$ satisfy
\begin{equation}
\label{eq: virtual vect}
W_{\R}-V_{\R}=H^{+}(E)\ \text{ in } KO(B),\qquad
V_{\C}-W_{\C}=\ind \slashed{D}_{E}\ \text{ in } K(B).
\end{equation}
Although this description is less precise, it is convenient to regard the families approximation \eqref{eq: fin. dim. approx. family} as a map of the form
\begin{equation}
\label{eq: fin. dim. approx. family2}
\begin{split}
\xymatrix{
f:\ \Th(\ind \slashed{D}_{E}) \ar@{->}[d] \ar[r] & \Th(H^{+}(E)) \ar@{->}[d] \\
B \ar@{=}[r] & B.
}
\end{split}
\end{equation}
Precisely speaking, in general there is no guarantee that we can ``desuspend'' the families approximation into this form.  
Another caveat is that $\ind \slashed{D}_{E}$ is a virtual vector bundle.
A map of Thom spaces is obtained after sufficient suspensions.

We now use gauge theory to constrain the characteristic classes and numbers of $H^{+}(E)$ and $\ind \slashed{D}_{E}$. In current applications to detecting non-smoothable families, only the \emph{torsion} information of $H^{+}(E)$ has been used, while from $\ind \slashed{D}_{E}$ only its rank—that is, the ordinary $\spc$ Dirac index on $X$—has been used. By the index theorem,
\begin{equation}
\label{eq: rank index bundle}
\rank_{\C}(\ind \slashed{D}_{E})=\bigl(c_{1}(\fraks)^{2}-\sigma(X)\bigr)/8.
\end{equation}

The following theorem of Baraglia is a families analogue of Donaldson's diagonalization theorem and serves as the key input for Theorem~\ref{theo Baraglia}, which asserts that for  ``most'' simply connected, closed $4$-manifolds $X$, the inclusion $\Diff(X) \inc \Homeo(X)$ is not a weak homotopy equivalence.

\begin{thm}[Baraglia~\cite{Ba21}]
\label{thm: Baraglia constraint}
Let $B$ be a compact topological space, let $X$ be an oriented smooth closed $4$-manifold, and let $\fraks$ be a $\spc$ structure on $X$. Let $(X,\fraks)\to E\to B$ be a smooth fiber bundle of $\spc$ $4$-manifolds. If the top Stiefel--Whitney class of $H^{+}(E)$ is nonzero, i.e. $w_{b^{+}(X)}(H^{+}(E))\neq 0$, then
\[
c_{1}(\fraks)^{2}-\sigma(X)\le 0.
\]
\end{thm}

\begin{rem}
\label{rem: local system}
As noted in \cite{Ba21}, the hypothesis $w_{b^{+}(X)}(H^{+}(E))\neq 0$ can be weakened to the nonvanishing of the Euler class with coefficients in a local system:
$e(H^{+}(E))\neq 0 \text{ in } H^{b^{+}(X)}(B;\Z_{w})$,
where $\Z_{w}$ is the local system on $B$ with fiber $\Z$ determined by $w=w_{1}(H^{+}(E))$.
\end{rem}

Before sketching the proof of Theorem~\ref{thm: Baraglia constraint}, let us explain why it may be viewed as a families version of Donaldson's diagonalization theorem. If $B=\{\mathrm{pt}\}$ and $b^{+}(X)=0$, the hypothesis is automatic, and Theorem~\ref{thm: Baraglia constraint} yields $c_{1}(\fraks)^{2}-\sigma(X)\le 0$ for \emph{every} $\spc$ structure $\fraks$ on $X$. Combined with Elkies' characterization of diagonalizable lattices~\cite{Elk95}, this implies that the intersection form of $X$ is diagonalizable over $\Z$. (All known derivations of Donaldson's theorem from Seiberg--Witten theory follow this pattern, relying on Elkies' result.)

\begin{proof}[Sketch of proof of Theorem~\ref{thm: Baraglia constraint}]
We sketch the argument in the case $b_{1}(X)=0$.
Consider the families finite-dimensional approximation \eqref{eq: fin. dim. approx. family2}, more precisely \eqref{eq: fin. dim. approx. family}. The map $f$ is $S^{1}$-equivariant. Taking $S^{1}$-fixed points in \eqref{eq: fin. dim. approx. family} yields a commutative diagram
\begin{equation}
\label{eq: fundamental diagram}
\begin{split}
\xymatrix{
\Th(V_{\R}\oplus V_{\C}) \ar[r]^{f} & \Th(W_{\R}\oplus W_{\C}) \\
\Th(V_{\R}) \ar[r]^{f^{S^{1}}} \ar@{->}[u] \ar@{->}[r] & \Th(W_{\R}) \ar@{->}[u]
}
\end{split}
\end{equation}
whose vertical maps are inclusions. Applying $H^{*}_{S^{1}}(-;\Z/2)$ gives a commutative diagram
\begin{equation}
\label{eq: fundamental diagram coh}
\begin{split}
\xymatrix{
H^{*}_{S^{1}}(\Th(V_{\R}\oplus V_{\C})) \ar@{->}[d] &
H^{*}_{S^{1}}(\Th(W_{\R}\oplus W_{\C})) \ar@{->}[d] \ar[l]_{f^{*}} \\
H^{*}_{S^{1}}(\Th(V_{\R})) &
H^{*}_{S^{1}}(\Th(W_{\R})) \ar[l]_{(f^{S^{1}})^{*}}
}
\end{split}
\end{equation}
(From now to the end of this subsection, cohomology is taken with $\Z/2$ coefficients.) Start with the $S^{1}$-equivariant Thom class
\[
\tau_{S^{1}}(W_{\R}\oplus W_{\C}) \in H^{*}_{S^{1}}(\Th(W_{\R}\oplus W_{\C}))
\]
in the upper right corner, and map it to the lower left corner along the two paths; commutativity gives an equality. Using (equivariant, $\Z/2$-coefficient) Thom isomorphisms repeatedly, one finds an element $\alpha\in H^{*}_{S^{1}}(B)$ such that
\begin{equation}
\label{eq: divisibility}
\alpha\, e_{S^{1}}(V_{\C}) = e_{S^{1}}(H^{+}(E))\, e_{S^{1}}(W_{\C}) \quad \text{in } H^{*}_{S^{1}}(B),
\end{equation}
where $e_{S^{1}}(-)$ denotes the $S^{1}$-equivariant Euler class with $\Z/2$ coefficients (here $\alpha$ is the degree of $f$ in the cohomology theory $H^{*}_{S^{1}}(-;\Z/2)$). The appearance of $H^{+}(E)$ in \eqref{eq: divisibility} comes from the fact that $f^{S^{1}}$ is induced fiberwise by an injective linear bundle map $f^{S^{1}}:V_{\R}\to W_{\R}$ and, by Hodge theory,
\[
\Coker\bigl(f^{S^{1}}:V_{\R}\to W_{\R}\bigr) \cong H^{+}(E).
\]
Since $S^{1}$ acts trivially on $B$ and on $H^{+}(E)$, we have
\[
e_{S^{1}}(H^{+}(E))= w_{b^{+}(X)}(H^{+}(E))\otimes 1 \in H^{*}(B)\otimes H^{*}_{S^{1}}(\mathrm{pt}) = H^{*}_{S^{1}}(B).
\]
Hence if $w_{b^{+}(X)}(H^{+}(E))\neq 0$, then $e_{S^{1}}(H^{+}(E))\neq 0$. 

Here recall that $H^{*}_{S^{1}}(\mathrm{pt})$ is a polynomial ring in one variable $U$.
In general, for a complex bundle $V\to B$ with the standard $S^{1}$-action by complex scalars on the fibers, the class $e_{S^{1}}(V)$ is a polynomial in $U$ with coefficients the Chern classes of $V$.
If one works with $\mathbb{Z}/2$-coefficients, Chern classes are replaced by Stiefel--Whitney classes.
Substituting this description for $e_{S^{1}}(V_{\C})$ and $e_{S^{1}}(W_{\C})$ into \eqref{eq: divisibility} and comparing top $U$-degrees using $e_{S^{1}}(H^{+}(E))\neq 0$, we obtain
\[
\rank_{\C} V_{\C} - \rank_{\C} W_{\C} \le 0 .
\]
Combining this with \eqref{eq: virtual vect} and \eqref{eq: rank index bundle} gives the desired inequality $c_{1}(\fraks)^{2}-\sigma(X)\le 0$.
\end{proof}

\begin{rem}
\label{rem: Baraglia constraint other cohomologies}
The preceding argument is a families version of Bauer-Furuta's derivation of Donaldson's diagonalization theorem from the finite-dimensional approximation~\cite{BF04}. When $(X,\fraks)$ is spin, the same reasoning goes through with $H^{*}_{S^{1}}(-;\Z/2)$ replaced by $H^{*}_{\Pin(2)}(-;\Z/2)$ or by $K_{\Pin(2)}$, yielding families analogues of Donaldson's Theorems B, C~\cite{Do86} and of Furuta's $10/8$-inequality~\cite{Fu01}; see \cite{Ba21}. (Using $H^{*}_{\Pin(2)}(-;\Z/2)$ improves the range of $n$ for which $\pi_{n}(\Diff(X))\to \pi_{n}(\Homeo(X))$ fails to be an isomorphism; cf. Remark~\ref{rem: Baraglia spin}.) In the $K_{\Pin(2)}$-setup one needs a $K$-theoretic orientation (a $\spc$ structure) on the bundle $H^{+}(E)\to B$ compatible with the $\Pin(2)$-action in order to apply Thom isomorphisms; this is a genuine obstruction to applications. Similarly, if one works with ordinary cohomology with $\Z$-coefficients, one needs an orientation on $H^{+}(E)$; this can be avoided either by working mod~2 as above or by using local coefficients as in Remark~\ref{rem: local system}. Existing applications only require the mod~2 argument, but it is an interesting problem whether using local systems yields further results.
\end{rem}

In contrast, the next theorem has no direct analogue in the unparametrized gauge theory; its proof uses an argument absent from the classical (nonfamilies) setting.

\begin{thm}[Baraglia--K.~\cite{BK19}]
\label{thm: BK Steenrod}
Let $B$ be a compact topological space and $X$ an oriented smooth closed $4$-manifold with $b_{1}(X)=0$ and $b^{+}(X)\equiv 3 \pmod 4$. Let $\fraks$ be a $\spc$ structure on $X$ with odd Seiberg--Witten invariant $SW(X,\fraks)$. Then for any smooth bundle $(X,\fraks)\to E\to B$ of $\spc$ 4-manifolds one has
\[
c_{1}(\ind \slashed{D}_{E}) \;=\; w_{2}(H^{+}(E)) \quad \text{in } H^{2}(B;\Z/2).
\]
In particular, for $X=K3$ with its standard spin structure, $w_{2}(H^{+}(E))=0$.
\end{thm}

The proof of Theorem~\ref{thm: BK Steenrod} applies Steenrod square operations to the families finite-dimensional approximation of the Seiberg--Witten equations to extract the stated constraint. This theorem is used in the proof of Theorem~\ref{theo: BK K3} comparing $\Diff(K3)$ and $\Homeo(K3)$.

\subsection{Constraints for families of 4-manifolds with boundary}
\label{subsection KT constraint}

Baraglia's families diagonalization theorem (Theorem~\ref{thm: Baraglia constraint}) extends to families of 4-manifolds with boundary. Recall that the Fr{\o}yshov invariant $\delta(Y,\frakt)$ is a numerical gauge-theoretic invariant associated to an oriented $\spc$ rational homology 3-sphere $(Y,\frakt)$. It can be defined via monopole Floer homology~\cite{KM07}, but also via the Seiberg--Witten Floer stable homotopy type or via Heegaard Floer homology (see Remark~\ref{rem: Froyshov relation correction term}); the Heegaard Floer definition allows for combinatorial computations for many 3-manifolds. Fr{\o}yshov introduced this invariant in his series of works beginning around 1996~\cite{Fr96,Fr02,Fr10} to extend Donaldson's diagonalization theorem to 4-manifolds with boundary (the case recovered below by taking $B=\{\rm pt\}$ and $X$ negative-definite). In the proof of Donaldson's theorem, reducible solutions (singular points of the moduli space) play a key role; heuristically, $\delta(Y)$ measures, via solutions to the 4-dimensional Seiberg--Witten equations on $Y\times \R$, how much flow there is from irreducible to reducible solutions on $Y$.

We now state the extension to families with boundary, which underlies Theorem~\ref{theo KT appl} on comparing diffeomorphism and homeomorphism groups for 4-manifolds with boundary.

\begin{thm}[K.--Taniguchi~\cite{KT20}]
\label{theo KT}
Let $X$ be an oriented smooth compact $4$-manifold with $b_{1}(X)=0$ and connected boundary $\partial X=Y$, a rational homology $3$-sphere. Let $\fraks$ be a spin$^{c}$ structure on $X$. Let $(X,\fraks)\to E\to B$ be a smooth bundle of $\spc$ $4$-manifolds over $B$ such that the restriction to the boundary is the trivial bundle of spin$^{c}$ $3$-manifolds. If $w_{b^{+}(X)}(H^{+}(E))\neq 0$, then
\begin{align*}
\bigl(c_{1}(\fraks)^{2}-\sigma(X)\bigr)/8 \;\le\; \delta(Y,\fraks|_{Y}).
\end{align*}
\end{thm}

Taking $B=\{\rm pt\}$ and $X$ negative-definite recovers Fr{\o}yshov's result~\cite{Fr96,Fr10}; taking $Y=S^{3}$ recovers Baraglia's Theorem~\ref{thm: Baraglia constraint}.

The proof of Theorem~\ref{theo KT} performs the families finite-dimensional approximation of the Seiberg--Witten equations in the presence of boundary. As a receptacle for this approximation we use Manolescu's Seiberg--Witten Floer stable homotopy type described in Subsection~\ref{subsection Seiberg--Witten Floer stable homotopy type}~\cite{Man03}. For families of spin 4-manifolds one can also work with $\mathrm{Pin}(2)$-equivariant cohomology to obtain sharper statements (see Remark~\ref{rem: Baraglia constraint other cohomologies}); in that case the Fr{\o}yshov invariant $\delta$ is replaced by Manolescu's invariants $\alpha,\beta,\gamma$~\cite{Man16}, among which $\beta$ is the one used in the disproof of the Triangulation Conjecture.

\begin{rem}
\label{rem: Froyshov relation correction term}
Since invariants equivalent to Fr{\o}yshov's appear in several Floer theories, we record their relationships. The Fr{\o}yshov invariant $h$ defined in monopole Floer homology~\cite{KM07}, the invariant $\delta$ defined from the Seiberg--Witten Floer stable homotopy type~\cite{Man16}, and the correction term $d$ in Heegaard Floer theory~\cite{OS03a} are all equivalent, related by
\[
\delta(Y,\frakt) \;=\; -\,h(Y,\frakt) \;=\; d(Y,\frakt)/2 .
\]
The equivalence between $\delta$ and $h$ is \cite[Corollary~1.3]{LM18}; the relationship between $h$ and $d$ is discussed, for example, in \cite[Remark~1.1]{LRS18}.
\end{rem}

\section{Other topics}
\label{section : whatelse}

We briefly touch on several other aspects of families gauge theory that we have not been able to discuss so far.

\subsection{$\Symp$ vs.\ $\Diff$}
\label{section: Diff vs. Symp}

Kronheimer~\cite{Kropre} compares the symplectomorphism group $\Symp(X,\omega)$ and the diffeomorphism group $\Diff(X)$ for a symplectic 4-manifold $(X,\omega)$ using families gauge theory. The basic tool is Taubes's result~\cite{Tau95,Tau96} that, for $\spc$ structures satisfying a certain relation with the cohomology class of the symplectic form, solutions to a suitably perturbed Seiberg--Witten equation (perturbed by the symplectic structure) vanish. Using this vanishing property, one constructs a cohomological invariant on the space of symplectic forms isotopic to $\omega$ by applying families Seiberg--Witten theory. The existence of a nontrivial homotopy class in the space of symplectic forms then reflects a difference between $\Symp(X,\omega)$ and $\Diff(X)$. To detect such nontrivial homotopy classes, one checks the nontriviality of the cohomological invariant by explicit computations for families arising from resolutions of singularities of algebraic surfaces. In the K\"{a}hler case, this reduces to the fact that solutions to the perturbed Seiberg--Witten equations correspond to algebraic curves.

Recently, Smirnov~\cite{Smi20,Smi202,Smi21,Smi23} has rapidly developed Kronheimer's approach.

\begin{rem}[Addendum after 2021]
Other related work by Lin~\cite{lin2022family} and Mu{\~n}oz{-}Ech{\'a}niz~\cite{munozechaniz2025configurationslagrangianspheresk3} has also appeared after 2021.
\end{rem}

\subsection{Families of Riemannian metrics}
\label{section: family of metrics}

Thus far we have considered families of 4-manifolds. There are also applications obtained by considering the \emph{trivial} family (a product bundle) and then studying families of Riemannian metrics with geometric origin on it. In that case, one obtains statements about the 4-manifold itself, rather than about bundles of 4-manifolds.

A basic example is Kronheimer--Mrowka's solution of the Thom conjecture~\cite{KM94}. This is a classical problem that bounds the genus of a smoothly embedded surface in $\CP^{2}$. The proof considers a 1-parameter family of metrics and exploits ``wall-crossing''. By stretching a neighborhood of an embedded surface, one obtains a 1-parameter family of metrics. The key point in the proof is that for a metric obtained by stretching a neighborhood of a surface whose genus is sufficiently small (more precisely, small enough to violate the adjunction inequality), the (unperturbed) Seiberg--Witten equations have no solutions.

Generalizing this idea, when several surfaces are embedded in a 4-manifold, one stretches their neighborhoods as independently as possible to obtain a higher-dimensional family of metrics. Studying the associated family of Seiberg--Witten equations yields constraints on configurations of surfaces (the author~\cite{K16,K17}). A construction of such families of metrics also appears in Fr{\o}yshov's work~\cite{Fr04} in the context of families of anti-self-dual equations. As a combination of diffeomorphisms with adjunction-type arguments, see also Baraglia~\cite{B202}.

An another instance, in the paper on the exact triangle by Kronheimer--Mrowka--Ozsv\'{a}th--Szab\'{o}~\cite{KMOS07} (and in F.~Lin's $\Pin(2)$-monopole version~\cite{FLin17}) one considers a 2-parameter family of metrics in a closely related spirit. While a 1-parameter family is typically what is needed (for invariance) in the construction of ordinary Floer homologies, taking 2-parameter families of geometrically derived metrics leads to computational formulas for Floer homology. Higher-parameter generalizations include the $A_{\infty}$-module structures in F.~Lin~\cite{FLin17a} and spectral sequences of Bloom~\cite{Blo11} and of Kronheimer--Mrowka~\cite{KM11}.

\subsection{Families of positive scalar curvature metrics}
\label{section: PSC}

The question of whether a given smooth manifold $X$ admits a (everywhere) positive scalar curvature metric is a classic problem in Riemannian geometry (see, for example, \cite{KW75} for its significance). A families version asks: assuming $X$ admits positive scalar curvature, what can be said about the homotopy groups of the space $\PSC(X)$ of positive scalar curvature metrics on $X$? In high dimensions this has been studied in considerable detail using surgery theory and index theory (e.g.\ \cite{BER-W17}), but in dimension four such surgery arguments do not work as well.

In dimension four, however, families Seiberg--Witten theory is effective. This is an extension to families of the fact that Seiberg--Witten equations have no solutions for positive scalar curvature metrics~\cite{W94}.
Ruberman~\cite{Rub01} first showed that there exists a 4-manifold $X$ with $\PSC(X)\neq\emptyset$ but $\pi_{0}(\PSC(X))\neq 0$ (concretely, $X$ is the form of $X=m\CP^{2}\# n\overline{\CP}^2$). This uses invariants of diffeomorphisms constructed from 1-parameter families of Seiberg--Witten equations. In \cite{K19}, using 2-parameter families of Seiberg--Witten equations, the author constructs nontrivial invariants for commuting pairs of diffeomorphisms and shows that $\PSC(X)$ can fail to be contractible even beyond the range accessible by Ruberman's method. (This invariant can be viewed as a special case of the invariants of \cite{LiLiu01,K21} explained in Section~\ref{section: how to consider families gauge theory}.)

\begin{rem}[Addendum after 2021]
Auckly and Ruberman \cite{auckly2025familiesdiffeomorphismsembeddingspositive} generalized Ruberman's result on $\pi_{0}(\PSC(X))$ to higher homotopy groups.
\end{rem}

As a somewhat different application, Baraglia and the author~\cite{BK20} show that families Seiberg--Witten theory provides obstructions to the existence of positive scalar curvature metrics \emph{invariant} under a given group action.

\bibliographystyle{alpha}
\bibliography{mainref}

\newcommand{\etalchar}[1]{$^{#1}$}
\begin{thebibliography}{KLMME24}

\bibitem[AB21]{AB21}
Mohammed Abouzaid and Andrew~J. Blumberg.
\newblock Arnold conjecture and {M}orava {K}-theory.
\newblock {\em arXiv:2103.01507}, 2021.

\bibitem[Akb15]{Akb14}
Selman Akbulut.
\newblock Isotoping 2-spheres in 4-manifolds.
\newblock In {\em Proceedings of the {G}\"{o}kova {G}eometry-{T}opology
  {C}onference 2014}, pages 264--266. G\"{o}kova Geometry/Topology Conference
  (GGT), G\"{o}kova, 2015.

\bibitem[AKM{\etalchar{+}}19]{AKMR}
Dave Auckly, Hee~Jung Kim, Paul Melvin, Daniel Ruberman, and Hannah Schwartz.
\newblock Isotopy of surfaces in 4-manifolds after a single stabilization.
\newblock {\em Adv. Math.}, 341:609--615, 2019.

\bibitem[AKMR15]{DKPR}
Dave Auckly, Hee~Jung Kim, Paul Melvin, and Daniel Ruberman.
\newblock Stable isotopy in four dimensions.
\newblock {\em J. Lond. Math. Soc. (2)}, 91(2):439--463, 2015.

\bibitem[AR25]{auckly2025familiesdiffeomorphismsembeddingspositive}
Dave Auckly and Daniel Ruberman.
\newblock Families of diffeomorphisms, embeddings, and positive scalar
  curvature metrics via {S}eiberg-{W}itten theory.
\newblock {\em arXiv:2501.11892}, 2025.

\bibitem[AS68]{AS68}
M.~F. Atiyah and I.~M. Singer.
\newblock The index of elliptic operators. {III}.
\newblock {\em Ann. of Math. (2)}, 87:546--604, 1968.

\bibitem[AS71]{AS71}
M.~F. Atiyah and I.~M. Singer.
\newblock The index of elliptic operators. {IV}.
\newblock {\em Ann. of Math. (2)}, 93:119--138, 1971.

\bibitem[Bar19]{B19}
David Baraglia.
\newblock Obstructions to smooth group actions on 4-manifolds from families
  {S}eiberg-{W}itten theory.
\newblock {\em Adv. Math.}, 354:106730, 32, 2019.

\bibitem[Bar21]{Ba21}
David Baraglia.
\newblock Constraints on families of smooth 4-manifolds from {B}auer-{F}uruta
  invariants.
\newblock {\em Algebr. Geom. Topol.}, 21(1):317--349, 2021.

\bibitem[Bar23a]{B21}
David Baraglia.
\newblock Non-trivial smooth families of {$K3$} surfaces.
\newblock {\em Math. Ann.}, 387(3-4):1719--1744, 2023.

\bibitem[Bar23b]{B20}
David Baraglia.
\newblock Tautological classes of definite 4-manifolds.
\newblock {\em Geom. Topol.}, 27(2):641--698, 2023.

\bibitem[Bar24]{B202}
David Baraglia.
\newblock An adjunction inequality obstruction to isotopy of embedded surfaces
  in 4-manifolds.
\newblock {\em Math. Res. Lett.}, 31(2):329--352, 2024.

\bibitem[Bau04a]{Ba04}
Stefan Bauer.
\newblock Refined {S}eiberg-{W}itten invariants.
\newblock In {\em Different faces of geometry}, volume~3 of {\em Int. Math.
  Ser. (N. Y.)}, pages 1--46. Kluwer/Plenum, New York, 2004.

\bibitem[Bau04b]{B04}
Stefan Bauer.
\newblock A stable cohomotopy refinement of {S}eiberg-{W}itten invariants.
  {II}.
\newblock {\em Invent. Math.}, 155(1):21--40, 2004.

\bibitem[BERW17]{BER-W17}
Boris Botvinnik, Johannes Ebert, and Oscar Randal-Williams.
\newblock Infinite loop spaces and positive scalar curvature.
\newblock {\em Invent. Math.}, 209(3):749--835, 2017.

\bibitem[BF04]{BF04}
Stefan Bauer and Mikio Furuta.
\newblock A stable cohomotopy refinement of {S}eiberg-{W}itten invariants. {I}.
\newblock {\em Invent. Math.}, 155(1):1--19, 2004.

\bibitem[BK20]{BK20}
David Baraglia and Hokuto Konno.
\newblock A gluing formula for families {S}eiberg-{W}itten invariants.
\newblock {\em Geom. Topol.}, 24(3):1381--1456, 2020.

\bibitem[BK22]{BK19}
David Baraglia and Hokuto Konno.
\newblock On the {B}auer-{F}uruta and {S}eiberg-{W}itten invariants of families
  of 4-manifolds.
\newblock {\em J. Topol.}, 15(2):505--586, 2022.

\bibitem[BK23]{BK21}
David Baraglia and Hokuto Konno.
\newblock A note on the {N}ielsen realization problem for {$K3$} surfaces.
\newblock {\em Proc. Amer. Math. Soc.}, 151(9):4079--4087, 2023.

\bibitem[BK26]{baraglia2024irreducible4manifoldsadmitexotic}
David Baraglia and Hokuto Konno.
\newblock Irreducible 4-manifolds can admit exotic diffeomorphisms.
\newblock {\em Duke Math. J.}, 175(4):717--733, 2026.

\bibitem[Blo11]{Blo11}
Jonathan~M. Bloom.
\newblock A link surgery spectral sequence in monopole {F}loer homology.
\newblock {\em Adv. Math.}, 226(4):3216--3281, 2011.

\bibitem[BS13]{BS13}
R.~\.{I}nan\c{c} Baykur and Nathan Sunukjian.
\newblock Round handles, logarithmic transforms and smooth 4-manifolds.
\newblock {\em J. Topol.}, 6(1):49--63, 2013.

\bibitem[CJS95]{CJS95}
R.~L. Cohen, J.~D.~S. Jones, and G.~B. Segal.
\newblock Floer's infinite-dimensional {M}orse theory and homotopy theory.
\newblock In {\em The {F}loer memorial volume}, volume 133 of {\em Progr.
  Math.}, pages 297--325. Birkh\"{a}user, Basel, 1995.

\bibitem[DK90]{DK90}
S.~K. Donaldson and P.~B. Kronheimer.
\newblock {\em The geometry of four-manifolds}.
\newblock Oxford Mathematical Monographs. The Clarendon Press, Oxford
  University Press, New York, 1990.
\newblock Oxford Science Publications.

\bibitem[Don83]{Do83}
S.~K. Donaldson.
\newblock An application of gauge theory to four-dimensional topology.
\newblock {\em J. Differential Geom.}, 18(2):279--315, 1983.

\bibitem[Don86]{Do86}
S.~K. Donaldson.
\newblock Connections, cohomology and the intersection forms of
  {$4$}-manifolds.
\newblock {\em J. Differential Geom.}, 24(3):275--341, 1986.

\bibitem[Don90]{Do90}
S.~K. Donaldson.
\newblock Polynomial invariants for smooth four-manifolds.
\newblock {\em Topology}, 29(3):257--315, 1990.

\bibitem[Elk95]{Elk95}
Noam~D. Elkies.
\newblock A characterization of the {${\bf Z}^n$} lattice.
\newblock {\em Math. Res. Lett.}, 2(3):321--326, 1995.

\bibitem[EO91]{EbelingOkonekdiffeomorphism91}
Wolfgang Ebeling and Christian Okonek.
\newblock On the diffeomorphism groups of certain algebraic surfaces.
\newblock {\em Enseign. Math. (2)}, 37(3-4):249--262, 1991.

\bibitem[FKM01]{FKM01}
M.~Furuta, Y.~Kametani, and N.~Minami.
\newblock Stable-homotopy {S}eiberg-{W}itten invariants for rational cohomology
  {$K3\#K3$}'s.
\newblock {\em J. Math. Sci. Univ. Tokyo}, 8(1):157--176, 2001.

\bibitem[FM88a]{FM88}
Robert Friedman and John~W. Morgan.
\newblock On the diffeomorphism types of certain algebraic surfaces. {I}.
\newblock {\em J. Differential Geom.}, 27(2):297--369, 1988.

\bibitem[FM88b]{FM882}
Robert Friedman and John~W. Morgan.
\newblock On the diffeomorphism types of certain algebraic surfaces. {II}.
\newblock {\em J. Differential Geom.}, 27(3):371--398, 1988.

\bibitem[FM12]{FM12}
Benson Farb and Dan Margalit.
\newblock {\em A primer on mapping class groups}, volume~49 of {\em Princeton
  Mathematical Series}.
\newblock Princeton University Press, Princeton, NJ, 2012.

\bibitem[Fre82]{Fre82}
Michael~Hartley Freedman.
\newblock The topology of four-dimensional manifolds.
\newblock {\em J. Differential Geometry}, 17(3):357--453, 1982.

\bibitem[Fr{\o}96]{Fr96}
Kim~A. Fr{\o}yshov.
\newblock The {S}eiberg-{W}itten equations and four-manifolds with boundary.
\newblock {\em Math. Res. Lett.}, 3(3):373--390, 1996.

\bibitem[Fr{\o}02]{Fr02}
Kim~A. Fr{\o}yshov.
\newblock Equivariant aspects of {Y}ang-{M}ills {F}loer theory.
\newblock {\em Topology}, 41(3):525--552, 2002.

\bibitem[Fr{\o}04]{Fr04}
Kim~A. Fr{\o}yshov.
\newblock An inequality for the {$h$}-invariant in instanton {F}loer theory.
\newblock {\em Topology}, 43(2):407--432, 2004.

\bibitem[Fr{\o}10]{Fr10}
Kim~A. Fr{\o}yshov.
\newblock Monopole {F}loer homology for rational homology 3-spheres.
\newblock {\em Duke Math. J.}, 155(3):519--576, 2010.

\bibitem[FU91]{FU91}
Daniel~S. Freed and Karen~K. Uhlenbeck.
\newblock {\em Instantons and four-manifolds}, volume~1 of {\em Mathematical
  Sciences Research Institute Publications}.
\newblock Springer-Verlag, New York, second edition, 1991.

\bibitem[Fur]{FurutaBonn}
Mikio Furuta.
\newblock Stable homotopy version of {S}eiberg-{W}itten invariant.
\newblock {\em Preprint}.

\bibitem[Fur01]{Fu01}
M.~Furuta.
\newblock Monopole equation and the {$\frac{11}{8}$}-conjecture.
\newblock {\em Math. Res. Lett.}, 8(3):279--291, 2001.

\bibitem[Fur02]{F02}
M.~Furuta.
\newblock Finite dimensional approximations in geometry.
\newblock In {\em Proceedings of the {I}nternational {C}ongress of
  {M}athematicians, {V}ol. {II} ({B}eijing, 2002)}, pages 395--403. Higher Ed.
  Press, Beijing, 2002.

\bibitem[GGH{\etalchar{+}}26]{gabai2023pseudoisotopies}
David Gabai, David Gay, Daniel Hartman, Vyacheslav Krushkal, and Mark Powell.
\newblock Pseudo-isotopies of simply connected 4-manifolds.
\newblock {\em Forum Math. Pi}, 14:Paper No. e9, 2026.

\bibitem[Gia09]{Gian09}
Jeffrey Giansiracusa.
\newblock The diffeomorphism group of a {$K3$} surface and {N}ielsen
  realization.
\newblock {\em J. Lond. Math. Soc. (2)}, 79(3):701--718, 2009.

\bibitem[GKT21]{GKT21}
Jeffrey Giansiracusa, Alexander Kupers, and Bena Tshishiku.
\newblock Characteristic classes of bundles of {K}3 manifolds and the {N}ielsen
  realization problem.
\newblock {\em Tunis. J. Math.}, 3(1):75--92, 2021.

\bibitem[Hat80]{Hat80}
A.~E. Hatcher.
\newblock Linearization in {$3$}-dimensional topology.
\newblock In {\em Proceedings of the {I}nternational {C}ongress of
  {M}athematicians ({H}elsinki, 1978)}, pages 463--468. Acad. Sci. Fennica,
  Helsinki, 1980.

\bibitem[HM74]{HM74}
Morris~W. Hirsch and Barry Mazur.
\newblock {\em Smoothings of piecewise linear manifolds}.
\newblock Annals of Mathematics Studies, No. 80. Princeton University Press,
  Princeton, N. J.; University of Tokyo Press, Tokyo, 1974.

\bibitem[IKMT25]{IKMT25}
Nobuo Iida, Hokuto Konno, Anubhav Mukherjee, and Masaki Taniguchi.
\newblock Diffeomorphisms of 4-manifolds with boundary and exotic embeddings.
\newblock {\em Math. Ann.}, 391(2):1845--1897, 2025.

\bibitem[KKN21]{KKN21}
Tsuyoshi Kato, Hokuto Konno, and Nobuhiro Nakamura.
\newblock Rigidity of the mod 2 families {S}eiberg-{W}itten invariants and
  topology of families of spin 4-manifolds.
\newblock {\em Compos. Math.}, 157(4):770--808, 2021.

\bibitem[KLMM24]{KLMME2}
Hokuto Konno, Jianfeng Lin, Anubhav Mukherjee, and Juan
  Mu{\~n}oz{-}Ech{\'a}niz.
\newblock The monodromy diffeomorphism of weighted singularities and
  {S}eiberg--{W}itten theory.
\newblock {\em arXiv:2411.12202}, 2024.

\bibitem[KLMME24]{KLMME}
Hokuto Konno, Jianfeng Lin, Anubhav Mukherjee, and Juan
  Mu^^c3^^b1oz-Ech^^c3^^a1niz.
\newblock On four-dimensional {D}ehn twists and {M}ilnor fibrations.
\newblock {\em arXiv:2409.11961}, 2024.
\newblock to appear in Duke Math. J.

\bibitem[KLS18]{KLS18}
Tirasan Khandhawit, Jianfeng Lin, and Hirofumi Sasahira.
\newblock Unfolded {S}eiberg-{W}itten {F}loer spectra, {I}: {D}efinition and
  invariance.
\newblock {\em Geom. Topol.}, 22(4):2027--2114, 2018.

\bibitem[KM94]{KM94}
P.~B. Kronheimer and T.~S. Mrowka.
\newblock The genus of embedded surfaces in the projective plane.
\newblock {\em Math. Res. Lett.}, 1(6):797--808, 1994.

\bibitem[KM02]{KM02}
Peter~B. Kronheimer and Ciprian Manolescu.
\newblock Periodic {F}loer pro-spectra from the {S}eiberg-{W}itten equations.
\newblock {\em arXiv:math/0203243}, 2002.

\bibitem[KM07]{KM07}
Peter Kronheimer and Tomasz Mrowka.
\newblock {\em Monopoles and three-manifolds}, volume~10 of {\em New
  Mathematical Monographs}.
\newblock Cambridge University Press, Cambridge, 2007.

\bibitem[KM11]{KM11}
P.~B. Kronheimer and T.~S. Mrowka.
\newblock Khovanov homology is an unknot-detector.
\newblock {\em Publ. Math. Inst. Hautes \'{E}tudes Sci.}, (113):97--208, 2011.

\bibitem[KM20]{KM20}
P.~B. Kronheimer and T.~S. Mrowka.
\newblock The {D}ehn twist on a sum of two {$K3$} surfaces.
\newblock {\em Math. Res. Lett.}, 27(6):1767--1783, 2020.

\bibitem[KMOS07]{KMOS07}
P.~Kronheimer, T.~Mrowka, P.~Ozsv\'{a}th, and Z.~Szab\'{o}.
\newblock Monopoles and lens space surgeries.
\newblock {\em Ann. of Math. (2)}, 165(2):457--546, 2007.

\bibitem[KMT23]{konno-mallick-taniguchi}
Hokuto Konno, Abhishek Mallick, and Masaki Taniguchi.
\newblock Exotic {D}ehn twists on 4-manifolds.
\newblock {\em arXiv:2306.08607}, 2023.
\newblock to appear in Geom. Topol.

\bibitem[KN23]{KN20}
Hokuto Konno and Nobuhiro Nakamura.
\newblock Constraints on families of smooth 4-manifolds from {${\rm
  Pin}^{-}(2)$}-monopole.
\newblock {\em Algebr. Geom. Topol.}, 23(1):419--438, 2023.

\bibitem[Kod64]{Kod64}
K.~Kodaira.
\newblock On the structure of compact complex analytic surfaces. {I}.
\newblock {\em Amer. J. Math.}, 86:751--798, 1964.

\bibitem[Kon]{KonnoRonsetsu2}
Hokuto Konno.
\newblock Diffeomorphism groups and gauge theory for families.
\newblock {\em S\=ugaku, to appear}.

\bibitem[Kon16]{K16}
Hokuto Konno.
\newblock Bounds on genus and configurations of embedded surfaces in
  4-manifolds.
\newblock {\em J. Topol.}, 9(4):1130--1152, 2016.

\bibitem[Kon19]{K19}
Hokuto Konno.
\newblock Positive scalar curvature and higher-dimensional families of
  {S}eiberg-{W}itten equations.
\newblock {\em J. Topol.}, 12(4):1246--1265, 2019.

\bibitem[Kon21]{K21}
Hokuto Konno.
\newblock Characteristic classes via 4-dimensional gauge theory.
\newblock {\em Geom. Topol.}, 25(2):711--773, 2021.

\bibitem[Kon22]{K17}
Hokuto Konno.
\newblock A cohomological {S}eiberg-{W}itten invariant emerging from the
  adjunction inequality.
\newblock {\em J. Topol.}, 15(1):108--167, 2022.

\bibitem[KPT26]{KangParkTaniguchi}
Sungkyung Kang, JungHwan Park, and Masaki Taniguchi.
\newblock Exotic {D}ehn twists and homotopy coherent group actions.
\newblock {\em Invent. Math.}, 243(1):209--241, 2026.

\bibitem[Kro]{Kropre}
Peter Kronheimer.
\newblock Some non-trivial families of symplectic structures.

\bibitem[KS77]{KS77}
Robion~C. Kirby and Laurence~C. Siebenmann.
\newblock {\em Foundational essays on topological manifolds, smoothings, and
  triangulations}.
\newblock Annals of Mathematics Studies, No. 88. Princeton University Press,
  Princeton, N.J.; University of Tokyo Press, Tokyo, 1977.
\newblock With notes by John Milnor and Michael Atiyah.

\bibitem[KT22]{KT20}
Hokuto Konno and Masaki Taniguchi.
\newblock The groups of diffeomorphisms and homeomorphisms of 4-manifolds with
  boundary.
\newblock {\em Adv. Math.}, 409:Paper No. 108627, 58, 2022.

\bibitem[KW75]{KW75}
Jerry~L. Kazdan and F.~W. Warner.
\newblock Scalar curvature and conformal deformation of {R}iemannian structure.
\newblock {\em J. Differential Geometry}, 10:113--134, 1975.

\bibitem[L\"98]{Lonnediffeomorphism98}
Michael L\"onne.
\newblock On the diffeomorphism groups of elliptic surfaces.
\newblock {\em Math. Ann.}, 310(1):103--117, 1998.

\bibitem[Lin17a]{FLin17a}
Francesco Lin.
\newblock {${\rm Pin}(2)$}-monopole {F}loer homology, higher compositions and
  connected sums.
\newblock {\em J. Topol.}, 10(4):921--969, 2017.

\bibitem[Lin17b]{FLin17}
Francesco Lin.
\newblock The surgery exact triangle in {${\rm Pin}(2)$}-monopole {F}loer
  homology.
\newblock {\em Algebr. Geom. Topol.}, 17(5):2915--2960, 2017.

\bibitem[Lin18]{FLin18}
Francesco Lin.
\newblock A {M}orse-{B}ott approach to monopole {F}loer homology and the
  triangulation conjecture.
\newblock {\em Mem. Amer. Math. Soc.}, 255(1221):v+162, 2018.

\bibitem[Lin22]{lin2022family}
Jianfeng Lin.
\newblock The family {S}eiberg-{W}itten invariant and nonsymplectic loops of
  diffeomorphisms.
\newblock {\em arXiv:2208.12082}, 2022.

\bibitem[Lin23]{JL20}
Jianfeng Lin.
\newblock Isotopy of the {D}ehn twist on {$K3\, \#\, K3$} after a single
  stabilization.
\newblock {\em Geom. Topol.}, 27(5):1987--2012, 2023.

\bibitem[LL01]{LiLiu01}
Tian-Jun Li and Ai-Ko Liu.
\newblock Family {S}eiberg-{W}itten invariants and wall crossing formulas.
\newblock {\em Comm. Anal. Geom.}, 9(4):777--823, 2001.

\bibitem[LM18]{LM18}
Tye Lidman and Ciprian Manolescu.
\newblock The equivalence of two {S}eiberg-{W}itten {F}loer homologies.
\newblock {\em Ast\'{e}risque}, (399):vii+220, 2018.

\bibitem[LM25]{LM21}
Jianfeng Lin and Anubhav Mukherjee.
\newblock Family {B}auer-{F}uruta invariant, exotic surfaces and {S}male
  conjecture.
\newblock {\em J. Assoc. Math. Res.}, 3(2):237--277, 2025.

\bibitem[LRS18]{LRS18}
Jianfeng Lin, Daniel Ruberman, and Nikolai Saveliev.
\newblock On the {F}r\o yshov invariant and monopole {L}efschetz number.
\newblock 2018.

\bibitem[LS14]{LS14}
Robert Lipshitz and Sucharit Sarkar.
\newblock A {K}hovanov stable homotopy type.
\newblock {\em J. Amer. Math. Soc.}, 27(4):983--1042, 2014.

\bibitem[LX23]{lin2023configuration}
Jianfeng Lin and Yi~Xie.
\newblock Configuration space integrals and formal smooth structures.
\newblock {\em arXiv:2310.14156}, 2023.

\bibitem[Man03]{Man03}
Ciprian Manolescu.
\newblock Seiberg-{W}itten-{F}loer stable homotopy type of three-manifolds with
  {$b_1=0$}.
\newblock {\em Geom. Topol.}, 7:889--932, 2003.

\bibitem[Man14]{Ma14}
Ciprian Manolescu.
\newblock On the intersection forms of spin four-manifolds with boundary.
\newblock {\em Math. Ann.}, 359(3-4):695--728, 2014.

\bibitem[Man16]{Man16}
Ciprian Manolescu.
\newblock Pin(2)-equivariant {S}eiberg-{W}itten {F}loer homology and the
  triangulation conjecture.
\newblock {\em J. Amer. Math. Soc.}, 29(1):147--176, 2016.

\bibitem[Mat86]{Matu86}
Takao Matumoto.
\newblock On diffeomorphisms of a {$K3$} surface.
\newblock In {\em Algebraic and topological theories ({K}inosaki, 1984)}, pages
  616--621. Kinokuniya, Tokyo, 1986.

\bibitem[Mil07]{Mil07}
John Milnor.
\newblock {\em Collected papers of {J}ohn {M}ilnor. {III}}.
\newblock American Mathematical Society, Providence, RI, 2007.
\newblock Differential topology.

\bibitem[Miy24]{miyazawaDehn}
Jin Miyazawa.
\newblock Boundary {D}ehn twists on {M}ilnor fibers and {F}amily
  {B}auer--{F}uruta invariants.
\newblock {\em arXiv:2410.21742}, 2024.

\bibitem[Moi52]{Moi52}
Edwin~E. Moise.
\newblock Affine structures in {$3$}-manifolds. {V}. {T}he triangulation
  theorem and {H}auptvermutung.
\newblock {\em Ann. of Math. (2)}, 56:96--114, 1952.

\bibitem[Mor96]{Mo96}
John~W. Morgan.
\newblock {\em The {S}eiberg-{W}itten equations and applications to the
  topology of smooth four-manifolds}, volume~44 of {\em Mathematical Notes}.
\newblock Princeton University Press, Princeton, NJ, 1996.

\bibitem[MS21]{MS21}
Ciprian Manolescu and Sucharit Sarkar.
\newblock A knot {F}loer stable homotopy type.
\newblock {\em arXiv:2108.13566}, 2021.

\bibitem[Mu{\~n}25]{munozechaniz2025configurationslagrangianspheresk3}
Juan Mu{\~n}oz{-}Ech{\'a}niz.
\newblock Configurations of {L}agrangian spheres in $k3$ surfaces.
\newblock {\em arXiv:2507.15039}, 2025.

\bibitem[MW09]{MW09}
Christoph M\"{u}ller and Christoph Wockel.
\newblock Equivalences of smooth and continuous principal bundles with
  infinite-dimensional structure group.
\newblock {\em Adv. Geom.}, 9(4):605--626, 2009.

\bibitem[Nak03]{Naka03}
Nobuhiro Nakamura.
\newblock The {S}eiberg-{W}itten equations for families and diffeomorphisms of
  4-manifolds.
\newblock {\em Asian J. Math.}, 7(1):133--138, 2003.

\bibitem[Nak10]{Naka10}
Nobuhiro Nakamura.
\newblock Smoothability of {$\Bbb Z\times\Bbb Z$}-actions on 4-manifolds.
\newblock {\em Proc. Amer. Math. Soc.}, 138(8):2973--2978, 2010.

\bibitem[Nak13]{Naka13}
Nobuhiro Nakamura.
\newblock {$\rm{Pin}^-(2)$}-monopole equations and intersection forms with
  local coefficients of four-manifolds.
\newblock {\em Math. Ann.}, 357(3):915--939, 2013.

\bibitem[Nak15]{Naka15}
Nobuhiro Nakamura.
\newblock {$\rm{Pin}^-(2)$}-monopole invariants.
\newblock {\em J. Differential Geom.}, 101(3):507--549, 2015.

\bibitem[Nic00]{Ni20}
Liviu~I. Nicolaescu.
\newblock {\em Notes on {S}eiberg-{W}itten theory}, volume~28 of {\em Graduate
  Studies in Mathematics}.
\newblock American Mathematical Society, Providence, RI, 2000.

\bibitem[OS03]{OS03a}
Peter Ozsv\'{a}th and Zolt\'{a}n Szab\'{o}.
\newblock Absolutely graded {F}loer homologies and intersection forms for
  four-manifolds with boundary.
\newblock {\em Adv. Math.}, 173(2):179--261, 2003.

\bibitem[Per86]{P86}
B.~Perron.
\newblock Pseudo-isotopies et isotopies en dimension quatre dans la
  cat\'{e}gorie topologique.
\newblock {\em Topology}, 25(4):381--397, 1986.

\bibitem[Qui86]{Q86}
Frank Quinn.
\newblock Isotopy of {$4$}-manifolds.
\newblock {\em J. Differential Geom.}, 24(3):343--372, 1986.

\bibitem[{Rad}25]{Ra1925}
T.~{Rad\'o}.
\newblock {\"Uber den Begriff der Riemannschen Fl\"ache.}
\newblock {\em {Acta Litt. Sci. Szeged}}, 2:101--121, 1925.

\bibitem[Rub98]{Rub98}
Daniel Ruberman.
\newblock An obstruction to smooth isotopy in dimension {$4$}.
\newblock {\em Math. Res. Lett.}, 5(6):743--758, 1998.

\bibitem[Rub99]{Rub99}
Daniel Ruberman.
\newblock A polynomial invariant of diffeomorphisms of 4-manifolds.
\newblock In {\em Proceedings of the {K}irbyfest ({B}erkeley, {CA}, 1998)},
  volume~2 of {\em Geom. Topol. Monogr.}, pages 473--488. Geom. Topol. Publ.,
  Coventry, 1999.

\bibitem[Rub01]{Rub01}
Daniel Ruberman.
\newblock Positive scalar curvature, diffeomorphisms and the {S}eiberg-{W}itten
  invariants.
\newblock {\em Geom. Topol.}, 5:895--924, 2001.

\bibitem[Rud16]{Rud16}
Yuli Rudyak.
\newblock {\em Piecewise linear structures on topological manifolds}.
\newblock World Scientific Publishing Co. Pte. Ltd., Hackensack, NJ, 2016.

\bibitem[San23]{S21}
Taketo Sano.
\newblock A {B}ar-{N}atan homotopy type.
\newblock {\em Internat. J. Math.}, 34(2):Paper No. 2350008, 45, 2023.

\bibitem[Smi20]{Smi202}
Gleb Smirnov.
\newblock Seidel's theorem via gauge theory.
\newblock {\em arXiv:2010.03361}, 2020.

\bibitem[Smi22a]{Smi20}
Gleb Smirnov.
\newblock From flops to diffeomorphism groups.
\newblock {\em Geom. Topol.}, 26(2):875--898, 2022.

\bibitem[Smi22b]{Smi21}
Gleb Smirnov.
\newblock Symplectic mapping class groups of {K}3 surfaces and
  {S}eiberg-{W}itten invariants.
\newblock {\em Geom. Funct. Anal.}, 32(2):280--301, 2022.

\bibitem[Smi23]{Smi23}
Gleb Smirnov.
\newblock Symplectic mapping class groups of blowups of tori.
\newblock {\em J. Topol.}, 16(3):877--898, 2023.

\bibitem[SS25]{SS21}
Hirofumi Sasahira and Matthew Stoffregen.
\newblock {\em Seiberg-{W}itten {F}loer spectra for {$b_1>0$}}, volume~17 of
  {\em Memoirs of the European Mathematical Society}.
\newblock EMS Press, Berlin, 2025.

\bibitem[Szy10]{Szy10}
Markus Szymik.
\newblock Characteristic cohomotopy classes for families of 4-manifolds.
\newblock {\em Forum Math.}, 22(3):509--523, 2010.

\bibitem[Tau95]{Tau95}
Clifford~Henry Taubes.
\newblock More constraints on symplectic forms from {S}eiberg-{W}itten
  invariants.
\newblock {\em Math. Res. Lett.}, 2(1):9--13, 1995.

\bibitem[Tau96]{Tau96}
Clifford~H. Taubes.
\newblock {${\rm SW}\Rightarrow{\rm Gr}$}: from the {S}eiberg-{W}itten
  equations to pseudo-holomorphic curves.
\newblock {\em J. Amer. Math. Soc.}, 9(3):845--918, 1996.

\bibitem[Wal64]{Wa64}
C.~T.~C. Wall.
\newblock Diffeomorphisms of {$4$}-manifolds.
\newblock {\em J. London Math. Soc.}, 39:131--140, 1964.

\bibitem[Wat18]{Wa18}
Tadayuki Watanabe.
\newblock Some exotic nontrivial elements of the rational homotopy groups of
  {$\mathrm{Diff}(S^{4})$}.
\newblock {\em arXiv:1812.02448}, 2018.

\bibitem[Wit94]{W94}
Edward Witten.
\newblock Monopoles and four-manifolds.
\newblock {\em Math. Res. Lett.}, 1(6):769--796, 1994.

\end{thebibliography}

\end{document}